\documentclass{amsart}

\calclayout
\usepackage{amssymb, amsmath}
\usepackage{mathrsfs}
\usepackage{mathtools}
\usepackage{amscd}
\usepackage{verbatim}

\usepackage{wasysym}

\usepackage[utf8]{inputenc}
\usepackage{enumerate}
\usepackage{url}

\usepackage[colorinlistoftodos,prependcaption,textsize=tiny]{todonotes}

\usepackage{cite}

\usepackage[colorlinks,linkcolor={blue},citecolor={blue},urlcolor={red},]{hyperref}

\allowdisplaybreaks

\usepackage{ulem}

\theoremstyle{plain}
\newtheorem{theorem}{Theorem}
\theoremstyle{remark}
\newtheorem{remark}[theorem]{Remark}
\newtheorem{example}[theorem]{Example}

\theoremstyle{plain}
\newtheorem{corollary}[theorem]{Corollary}
\newtheorem{lemma}[theorem]{Lemma}
\newtheorem{proposition}[theorem]{Proposition}
\newtheorem{definition}[theorem]{Definition}

\newtheorem{assumption}[theorem]{Assumption}
\newtheorem{assumptions}[theorem]{Assumptions}

\numberwithin{theorem}{section}

\numberwithin{equation}{section}

\def\N{{\mathbb N}}

\def\R{{\mathbb R}}

\newcommand{\E}{{\mathbb E}}
\renewcommand{\P}{{\mathbb P}}

\newcommand{\F}{{\mathscr F}}

\newcommand{\D}{{\mathcal D}}
\newcommand{\calL}{{\mathcal L}}

\newcommand{\one}{{{\bf 1}}}

\newcommand{\lb}{\langle}
\newcommand{\rb}{\rangle}

\newcommand{\T}{\mathbb{T}}
\newcommand{\wt}{\widetilde}
\renewcommand{\emptyset}{\varnothing}

\newcommand {\Distr}{\mathcal{D}}

\newcommand{\Tr}{{\rm Tr}}

\newcommand{\dint}{\mathrm{d}}
\newcommand{\fterm}{\E\Big(\int_0^T f(s) \dint s \Big)^{\frac{p}{2}}}

\newcommand{\condref}[1]{\textnormal{({\color{blue}H\ref{#1}})}}

\DeclareMathOperator{\DIV}{div}

\usepackage{xcolor}

\begin{document}

\title[Higher order moments for SPDE with monotone nonlinearities]{Higher order moments for SPDE with monotone nonlinearities}

\author{Manuel V. Gnann}
\address[Manuel V. Gnann]{Delft Institute of Applied Mathematics\\
    Delft University of Technology \\ P.O. Box 5031\\ 2600 GA Delft\\The
    Netherlands} \email{M.V.Gnann@tudelft.nl}

\author{Jochem Hoogendijk}
\address[Jochem Hoogendijk]{Mathematical institute\\
    Utrecht University \\ P.O. Box 80010 \\ 3508 TA Utrecht\\ The
    Netherlands} \email{j.p.c.hoogendijk@uu.nl}

\author{Mark C. Veraar}
\address[Mark C. Veraar]{Delft Institute of Applied Mathematics\\
    Delft University of Technology \\ P.O. Box 5031\\ 2600 GA Delft\\The
    Netherlands} \email{M.C.Veraar@tudelft.nl}

\thanks{This work is based on the second author's master's thesis in applied mathematics, prepared under the supervision of the first and third author at Delft University of Technology.}

\date\today



\begin{abstract}
    This paper introduces a new $p$-dependent coercivity condition through which $L^p$-moments for solutions can be obtained for a large class of SPDEs in the variational framework. If $p=2$, our condition reduces to the classically coercivity condition, which only yields second moments for the solution. The abstract result is shown to be optimal. Moreover, the results are applied to obtain $L^p$-moments of solutions for several classical SPDEs such as stochastic heat equations with Dirichlet and Neumann boundary conditions, Burgers' equation and the Navier-Stokes equations in two spatial dimensions. Furthermore, we can recover recent results for systems of SPDEs and higher order SPDEs using our unifying coercivity condition.
\end{abstract}
\maketitle

\section{Introduction}
In this paper we introduce a new coercivity condition through which one can obtain estimates for higher order moments for stochastic partial differential equations (SPDEs) of the form
\begin{equation}\label{eq:variationalintro}
    \dint u(t) = A(t, u(t)) \dint t + B(t, u(t)) \dint W(t), \qquad u(0) = u_0.
\end{equation}
Here $W$ is a $U$-cylindrical Brownian motion. We will be concerned with the so-called {\em variational} or {\em monotone operator approach} to SPDEs in Hilbert spaces. In particular, we assume that $(V, H, V^*)$ is a Gelfand triple, where $H$ is a separable Hilbert space and $V$ a reflexive Banach space.

The variational approach for SPDEs was introduced in 1972 by Bensoussan and Temam using time discretizations methods \cite{bensoussan_equations_1972}. Pardoux improved the latter via Lions' approach for PDEs in \cite{pardoux_equations_1975}. In this approach, Galerkin approximations are used together with a priori energy estimates to obtain existence and uniqueness. Since then, both Krylov and Rozovskii \cite{krylov_1979} and Liu and R\"{o}ckner \cite{rockner_2010, Rockner_SPDE_2015} have extended this approach even further by allowing monotone and locally monotone operators, respectively, as the driving part of the equation.

An advantage of the variational approach is that it directly applies to nonlinear equations. Another key property is that it typically gives global existence and uniqueness at once, and there is often no need to check any further blow-up criteria for the solution. When combined with other approaches this can be very effective (see e.g.\ \cite{AV20_NS} for the stochastic Navier-Stokes equations).

Each of the above papers assumes a coercivity condition on $(A,B)$ of the form (see Section~\ref{sec:main} for explanation on the notation):
\begin{equation}\label{eq:coercivityintro}
    2\langle A(t, v), v \rangle + \|B(t, v)\|_{\calL_2(U, H)}^2 \leq -\theta\|v\|_V^2 +  K\|v\|_H^2 + f(t).
\end{equation}
Note that $B(t, \cdot)$ is allowed to be defined on the smallest space $V$. In the above mentioned results for the variational approach to \eqref{eq:variationalintro} one obtains estimates for
\begin{equation}\label{eq:toestpintro}
    \E\sup_{t\in [0,T]} \|u(t)\|^p_H \ \ \text{and} \ \ \E\|u\|_{L^2(0,T;V)}^p,
\end{equation}
but only for $p=2$. Estimates for $p>2$ are not available unless $B$ is assumed to be defined on $H$ instead of $V$ (see \cite[Section 5]{Rockner_SPDE_2015}).
An attempt to treat more general $p\geq 2$ (and even $p<2$) was made in \cite{veraar_2012} by Brze\'{z}niak and the third author. Here it also turned out that the classical coercivity condition is not strong enough to obtain finite $L^p$-moments. The paper \cite{veraar_2012} only considers a simplified setting. Therefore, it was enlightening to see that in \cite{neelima_2020} by Neelima and \v{S}i\v{s}ka, some results can be proved in a general monotone setting. However, the $L^p$-bounds proved there are only sub-optimal (see Remark~\ref{rem:comparison} for details), and the coercivity condition they used seems too restrictive in some cases, which becomes clear further below and in the presented applications.

In the current paper we obtain a complete generalization of the classical monotone operator framework leading to estimates for \eqref{eq:toestpintro} for $p>2$. From \cite{veraar_2012} it follows that the terms in \eqref{eq:toestpintro} are infinite for $p>2$. Therefore, a restriction is necessary. The key ingredient turns out to be the following $p$-dependent coercivity condition:
\begin{equation}\label{eq:coercivityintrop}
\begin{aligned}
    2\langle A(t, v), v \rangle + \|B(t, v)\|^2_{\calL_2(U, H)} + (p-2)&\frac{\|B(t, v)^*v\|_U^2}{\|v\|_H^2} \\ & \leq -\theta\|v\|_V^\alpha + K_c\|v\|_H^2 + f(t).
\end{aligned}
\end{equation}
Our main result (Theorem~\ref{Main_theorem}) states that under \eqref{eq:coercivityintrop} and the usual conditions in the monotone operator framework, one can estimate the norms in \eqref{eq:toestpintro}. Note that \eqref{eq:coercivityintrop} reduces to \eqref{eq:coercivityintro} if $p=2$. In Example~\ref{ex:optimal} we use a specific choice suggested in \cite{veraar_2012} to show that \eqref{eq:coercivityintrop} is optimal. The proof of the main result is elementary, but quite tedious. In some cases we give explicit constants in the obtained estimates for the moments.

An interesting special case occurs if $B(t, v)^*v = 0$, since then the $p$-dependent term in \eqref{eq:coercivityintrop} vanishes and we get estimates for all $p\geq 2$. This typically occurs for differential operators of odd order with suitable boundary conditions. In some cases we can even let $p\to \infty$ to obtain uniform estimates in $\Omega$.

In Section~\ref{sec:appl} we consider applications to the stochastic heat equation with Dirichlet and Neumann boundary conditions, Burgers' equation, the stochastic Navier-Stokes equations  in dimension two, systems, higher order equations, and the $p$-Laplace equation.

\section{Setting and main result}\label{sec:main}
Before we state our main result we fix our notation and terminology. For further details on Gelfand triples and stochastic integration theory we refer to \cite{Rockner_SPDE_2015}.

Throughout this paper $(U, (\cdot, \cdot)_U)$ and $(H, (\cdot, \cdot)_H)$ denote real separable Hilbert spaces and $(V, \|\cdot\|_V)$ is a reflexive Banach space embedded continuously and densely in $H$. The dual of $V$ (relative to $H$) is denoted by $V^*$ and the duality pairing between $V$ and $V^*$ by $\langle \cdot, \cdot \rangle$. The probability space $(\Omega, \mathcal{A},\P)$ and filtration $(\F_t)_{t\geq 0}$ will be fixed. The progressive $\sigma$-algebra is denoted by $\mathcal{P}$. Furthermore, suppose that $(W(t))_{t\geq 0}$ is a $U$-cylindrical Brownian motion with respect to $(\F_t)_{t\geq 0}$.

\subsection{Assumptions}

The main assumptions on the nonlinearities are as follows:

\begin{assumptions}\label{main_assumptions}
    Let
    \[A:[0, T] \times \Omega \times V \to V^*, \ \ \text{and} \ \ B: [0, T] \times \Omega \times V \to \calL_2(U, H) \]
    both be $\mathcal{P}\otimes \mathcal{B}(V)$-measurable.  Suppose that there exist finite constants
    \[\alpha > 1, \ \beta \geq 0, \ p \geq \beta + 2, \ \theta > 0, \ K_c, K_A, K_B, K_\alpha \geq 0\]
    and $f \in L^\frac{p}{2}(\Omega; L^1([0, T]))$ such that for all $t\in[0, T]$ a.s.
    \begin{enumerate}[(H1)]
        \item\label{it:hem} (Hemicontinuity) For all  $u, v, w \in V$, $\omega \in \Omega$, the following map is continuous:
              \[\lambda \mapsto \langle A(t, u+\lambda v, \omega), w \rangle.\]
        \item\label{it:weak_mon} (Local weak monotonicity) For all $u, v \in V$,
        \begin{align*}      
            &2\langle A(t, u)-A(t, v), u-v\rangle + \|B(t, u) - B(t, v) \|^2_{\calL_2(U, H)}\\
            &\quad \leq K (1+\|v\|_V^\alpha)(1+\|v\|_H^\beta)\|u-v\|_H^2 .
        \end{align*}
        \item\label{it:coerc} (Coercivity) For all $v \in V$, $v\neq 0$, $$2\langle A(t, v), v \rangle + \|B(t, v)\|^2_{\calL_2(U, H)} +(p-2)\frac{\|B(t, v)^*v\|_U^2}{\|v\|_H^2} \leq -\theta\|v\|_V^\alpha + f(t) + K_c\|v\|_H^2.$$
        \item\label{it:bound1} (Boundedness 1) For all $v\in V$, $$\|A(t, v)\|_{V^*}^{\frac{\alpha}{\alpha-1}}\leq K_A(f(t)+\|v\|_V^{\alpha})(1+\|v\|_H^\beta).$$
        \item\label{it:bound2} (Boundedness 2) For all $v \in V$,
              $$\|B(t, v)\|_{\calL_2(U, H)}^2 \leq f(t)+K_B\|v\|_H^2 + K_\alpha\|v\|_V^\alpha.$$
    \end{enumerate}
\end{assumptions}
Most conditions are standard and appear in previous works that treat the variational approach to SPDEs (see \cite{pardoux_equations_1975,krylov_stochastic_1981,Rockner_SPDE_2015}).
Usually, in these works $\beta = 0$. The case $\beta\geq 0$ is considered in \cite{brzezniak_strong_2014} where L\'{e}vy noise is treated as well. The condition $p\geq \beta+2$ is needed for a priori bounds for involving the $L^{\frac{\alpha}{\alpha-1}}(\Omega\times[0,T])$-norm of $\|A(t,u(t))\|$ for $u\in L^{p}(\Omega;C([0,T];H))\cap L^{\frac{p \alpha}{2}}(\Omega;L^{\alpha}([0,T];V))$ needed in the existence proof. Often it can be avoided by a localization argument.
Our hypothesis \condref{it:coerc} is new and will allow us to obtain estimates for $L^p$-moments. It reduces to the classical coercivity assumption if $p=2$. The function $f$ can be used to include inhomogeneous terms in $A$ and $B$.

After these preparations we can define solutions to \eqref{eq:variationalintro}.
\begin{definition}\label{solution_definition}
    Suppose that Assumptions~\ref{main_assumptions} hold and let $u(0):\Omega\to H$ be $\F_0$-measurable. An adapted, continuous $H$-valued process $u$ is called a {\em solution} to \eqref{eq:variationalintro} if $u\in L^{\alpha}(0,T;V)$ a.s.\ and
    for every $t\in[0, T]$, a.s., $$u(t) = u(0) + \int_0^t A(s, u(s)) \dint s + \int_0^t B(s, u(s)) \dint W(s).$$
\end{definition}
Note that due to \condref{it:bound1}, $t\mapsto A(t, u(t))\in L^{\frac{\alpha}{\alpha-1}}(0,T;V^*)$ a.s.\ and thus the above Bochner integral is well-defined. Due to \condref{it:bound2}, $t\mapsto B(t, u(t))\in L^2(0,T;\calL_2(U,H))$ a.s.\ and thus the stochastic integral is also well-defined.

The following can be checked by elementary arguments involving Young's inequality and inequalities for convex functions:
\begin{remark}\label{rem:additive}
    Let $\phi\in L^{\frac{p\alpha}{2(\alpha-1)}}(\Omega;L^{\frac{\alpha}{\alpha-1}}(0,T;V^*))$ and $\psi\in L^{p}(\Omega;L^2(0,T;\calL_2(U,H)))$
    If $(A,B)$ satisfies Assumptions~\ref{main_assumptions}, then $(A+\phi, B+\psi)$ satisfies Assumptions~\ref{main_assumptions} with the same $\alpha, \beta$, $p$, and $f$ replaced by
    \[\tilde{f} = f + \|\phi\|^{\frac{\alpha}{\alpha-1}}_{V^*} + \|\psi\|_{\calL_2(U,H)}^2.\]
\end{remark}

\subsection{Main result}
The main result of this paper is the following well-posedness result with higher order moments:
\begin{theorem}\label{Main_theorem}
    Suppose that Assumptions~\ref{main_assumptions} hold and let $u(0) \in L^{p}(\Omega,\F_0;H)$. Then \eqref{eq:variationalintro} has a unique solution $u$, and there exists a constant $C$ depending on $\alpha$, $\beta$, $\theta$, $p$, $K_c$, $K_A$, $K_B$, $K_{\alpha}$ such that
    \begin{equation}\label{eq:aprioripnew}
        \begin{split}
            \E\sup\limits_{t\in[0, T]}\|u(t)\|_H^p + \E\Big(\int_0^T \|u(t)\|_V^\alpha \dint t\Big)^{\frac{p}{2}}\leq Ce^{CT}\Big[\E\|u(0)\|_H^p + \E\Big(\int_0^T f(t) \dint t\Big)^{\frac{p}{2}}\Big].
        \end{split}
    \end{equation}
\end{theorem}

The proof is given in Section~\ref{sec:proofMain}. The main novelty is the a priori estimate \eqref{eq:aprioripnew}. The existence and uniqueness can be obtained by standard Galerkin approximation techniques.
In Corollary~\ref{cor:a_priori_remark} in case $K_B = K_c = 0$, the $p$-dependence in the estimate \eqref{eq:aprioripnew} will be made explicit.

The following example is taken from \cite{veraar_2012} and implies optimality of Theorem~\ref{Main_theorem} with respect to $p$ in the sense that if $p$ is replaced by some number $q>p$, then it can happen that $\E\|u(t)\|_H^{q}=\infty$.
\begin{example}[Optimality]\label{ex:optimal}
    On the torus $\T$ consider the equation
    \begin{equation}\label{veraar_example}
        \dint u(t) = \Delta u(t) \dint t + 2 \gamma (-\Delta)^{\frac{1}{2}} u(t) \dint W(t), \qquad u(0) = u_0.
    \end{equation}
    Here $\gamma \in \mathbb{R}$, $u_0\in L^{p}(\Omega,\F_0;L^2(\T))$ and $W$ is a real-valued Wiener process (thus $U = \R$). In \cite{veraar_2012} it is proved that \eqref{veraar_example} has a unique solution in $L^{p}(\Omega;L^2(0,T; H^1(\T)))$ if $2\gamma^2(p-1) < 1$. Indeed, setting $V = H^1(\T)$, $H = L^2(\T)$, $A = \Delta$, and $B = 2\gamma(-\Delta)^{1/2}$, Assumptions \ref{main_assumptions} \condref{it:hem}, \condref{it:weak_mon}, \condref{it:bound1}, \condref{it:bound2}  hold with $\alpha = 2$, $\beta = 0$ and $f=0$ and suitable constants $K$, $K_A$ and $K_B$. To check \condref{it:coerc} note that
  \begin{align*}
   2\lb \Delta v, v\rb+ \|B(v)\|_{L^2(\T)}^2 +  (p-2) \frac{|B(v)^*v|^2}{\|v\|_{L^2(\T)}^2}
         & \leq 2\lb \Delta v, v\rb+ (p-1)\|B(v)\|_{L^2(\T)}^2
        \\ & \leq -2 \|\nabla v\|_{L^2(\T)}^2 +  4\gamma^2(p-1)\|v\|^2_{H^1(\T)}\\ & \leq  -\theta\|v\|^2_{H^1(\T)} + 2 \|v\|_{L^2(\T)}^2,
    \end{align*}
    where $\theta:=2-4\gamma^2(p-1)>0$. This proves \condref{it:coerc} and thus the well-posedness follows from Theorem~\ref{Main_theorem}.
    On the other hand, it follows from \cite[Theorem~4.1(ii)]{veraar_2012} that there exists an initial datum $u_0\in C^\infty(\T)$ such that if $q>p$ and $\gamma>0$ is such that $2\gamma^2(p-1) < 1$ and $2\gamma^2(q-1) > 1$, then there is a $t>0$ such that $\E\|v(t)\|_{L^2(\T)}^{q} = \infty$. Moreover, even $\E\|v(t)\|_{H^{s}(\T)}^{q} = \infty$ for all $s\in \R$.
\end{example}

\begin{remark}\label{rem:comparison}
    In \cite{neelima_2020} the following coercivity condition was proposed:
    \begin{align}\label{eq:coercivitypmin1}
    2\langle A(t, v), v \rangle + (p-1)\|B(t, v)\|^2_{\calL_2(U, H)}\leq -\theta\|v\|_V^\alpha + f(t) + K_c\|v\|_H^2, \ \ v\in V.
    \end{align}
    The latter is more restrictive than \condref{it:coerc}, since $\frac{\|B(t, v)^*v\|_U^2}{\|v\|_H^2}\leq \|B(t, v)\|^2_{\calL_2(U, H)}$. Replacing our condition \condref{it:coerc} by \eqref{eq:coercivitypmin1}, the main result in \cite{neelima_2020} states that
    \begin{align*}
        \sup_{t\in [0,T]} \E\|u(t)\|^{p} & \leq C \Big[\E\|u(0)\|_H^{p} + \E\Big(\int_0^T f(t) \dint t\Big)^{\frac{p}{2}}\Big],
        \\ \E\sup_{t\in [0,T]}  \|u(t)\|^{r p} &\leq C_r \Big[\E\|u(0)\|_H^{p} + \E\Big(\int_0^T f(t) \dint t\Big)^{\frac{p}{2}}\Big],
    \end{align*}
    where $r\in (0,1)$.
    Both estimates are sub-optimal. The result \eqref{eq:aprioripnew} shows that the supremum can actually be inside the expectation and thus one can take $r=1$. In \cite{neelima_2020} the growth condition \condref{it:bound2} on $B$ is not explicitly assumed, but as far as we can see \condref{it:bound2} is used in their estimate (13).

    Similar results were obtained in \cite{brzezniak_strong_2014}, under a  different coercivity condition. A detailed comparison with \eqref{eq:coercivitypmin1} can be found in \cite[Remark 6.1]{neelima_2020}. 
    \end{remark}

\section{Proof of the main result}\label{sec:proofMain}
In \cite[Theorem~4.2.5, p. 91]{Rockner_SPDE_2015} the following version of It\^o's  formula is obtained for $p=2$. The $p>2$ version can be obtained from the $p=2$ version combined with
the real case by considering $(\|X_t\|^2+\varepsilon)^{p/2}$ and letting $\varepsilon\downarrow 0$ or by applying \cite[Theorem~3.2, p.~73]{rozovskii_1990}.
\begin{lemma}[It\^{o}'s formula for $\|\cdot\|_H^p$]\label{ito_p}
    Let $p\in [2, \infty)$, $\alpha \in (1, \infty)$, $X_0 \in L^p(\Omega; \mathcal{F}_0; H)$ and $Y\in L^{\frac{\alpha}{\alpha-1}}([0, T]\times \Omega; \dint t \otimes \mathbb{P};V^*)$, $Z\in L^2([0, T]\times \Omega; \dint t \otimes \mathbb{P}; \calL_2(U, H))$ both progressively measurable. If $X \in L^\alpha([0, T]\times \Omega; \dint t \otimes \mathbb{P}; V)$ and for a.e.\ $t\in[0, T]$ $\E(\|X_t\|_H^2) < \infty$, and a.s.
    $$X_t = X_0 + \int_0^t Y_s \dint s + \int_0^t Z_s \dint W_s, \quad t\in[0, T]$$
    is satisfied in $V^*$, then X is a continuous $H$-valued $\mathcal{F}_t$-adapted process and the following holds a.s.:
    \begin{align*}
        \|X_t\|_H^p & = \|X_0\|_H^p + p\int_0^t \|X_s\|_H^{p-2} Z_s^* X_s \dint W_s                                                                                 \\
                    & \quad + \frac{p(p-2)}{2}\int_0^t \|X_s\|_H^{p-4} \|Z_s^* X_s\|_U^2 \dint s                                                                    \\
                    & \quad + \frac{p}{2}\int_0^t \|X_s\|_H^{p-2} \left(2\langle Y_s, X_s\rangle + \|Z_s\|_{\calL_2(U, H)}^2 \right)\dint s, \quad t\in[0, T],
    \end{align*}
    where $\|X_s\|_H^{p-4}$ is defined as zero if $X_s=0$.
\end{lemma}

The main step in the proof of Theorem~\ref{Main_theorem} is the following new a priori estimate, where we note that the condition $p\geq \beta+2$ in Assumptions~\ref{main_assumptions} is not needed.
\begin{theorem}\label{a_priori_theorem}
    Suppose $u$ is a solution of equation \eqref{eq:variationalintro} with initial condition $u(0)\in L^{p}(\Omega; H)$ and \condref{it:coerc}, \condref{it:bound1} and \condref{it:bound2} from Assumptions~\ref{main_assumptions} hold with $f\in L^\frac{p}{2}(\Omega; L^1([0, T]))$. Then, there exists a constant $C$ depending on $\alpha$, $\beta$, $\theta$, $p$, $K_c$, $K_A$, $K_B$, $K_{\alpha}$ such that
    \begin{equation}\label{eq:estapriorimain}
        \begin{split}
            \E\sup\limits_{t\in[0, T]}\|u(t)\|_H^p + \E\Big(\int_0^T \|u(t)\|_V^\alpha \dint t\Big)^{\frac{p}{2}}
            &\leq Ce^{CT}\Big[\E\|u(0)\|_H^p + \E\Big(\int_0^T f(t) \dint t\Big)^{\frac{p}{2}}\Big].
        \end{split}
    \end{equation}
\end{theorem}

\begin{proof}
    \textit{Step 0: Stopping time argument}.

    For $n\geq 1$ consider the following sequence of stopping times:
    \begin{equation*}
        \tau_n = \inf\{t\in[0, T] : \|u(t)\|_H \geq n\} \wedge \inf\{t\in[0, T]: \int_0^t \|u(s)\|_V^\alpha ds \geq n \},
    \end{equation*}
    where we set $\inf \emptyset = T$. Then $\tau_n \to T$ a.s. as $n\to \infty$ by Definition~\ref{solution_definition}. Since $u$ solves \eqref{eq:variationalintro} in the sense of Definition~\ref{solution_definition}, Lemma~\ref{ito_p} implies the following:
    \begin{align*}
         & \|u(t\wedge \tau_n)\|_H^{p}  = \|u(0)\|_H^{p}  + p\int_0^{t\wedge\tau_n} \|u(s)\|_H^{p-2}B(s, u(s))^* u(s) \dint W(s)                                                                          \\
         & + \frac{p}{2}\int_0^{t\wedge\tau_n} \|u(s)\|_H^{p-2} \Big(2\langle A(s, u(s)), u(s)\rangle +  \|B(s, u(s))\|_{\calL_2(U, H)}^2\\
         &\quad + (p-2)\frac{\|B(s, u(s))^*u(s)\|_U^2}{\|u(s)\|_H^2}\Big) \dint s.
    \end{align*}
    Using the coercivity assumption \condref{it:coerc}, the latter implies
    \begin{equation}\label{ItoIneq}
        \begin{split}
            \|u(t\wedge \tau_n)\|_H^{p} &+\frac{\theta p}{2}\int_0^{t\wedge\tau_n}\|u(s)\|_H^{p-2} \|u(s)\|_V^\alpha \dint s \\
            & \leq \|u(0)\|_H^{p} + p\int_0^{t\wedge\tau_n} \|u(s)\|_H^{p-2} B(s, u(s))^*u(s) \dint W(s) \\
            & \phantom{\leq} + \frac{p}{2}\int_0^{t\wedge\tau_n} \|u(s)\|_H^{p-2} \left(f(s)+K_c\|u(s)\|_H^2 \right) \dint s.\\
        \end{split}
    \end{equation}
    Taking expectations in \eqref{ItoIneq}, the stochastic integral cancels and we find
    \begin{equation}\label{apriori}
        \begin{split}
            \E\|u(t\wedge\tau_n)&\|_H^p +\frac{\theta p}{2}\E\int_0^{t\wedge\tau_n}\|u(s)\|_H^{p-2}\|u(s)\|_V^\alpha \dint s\\
            & \leq \E\|u(0)\|_H^p +\frac{p}{2}\E\int_0^{t\wedge\tau_n}\|u(s)\|_H^{p-2}f(s) \dint s
            +\frac{p}{2}K_c\E\int_0^{t\wedge\tau_n}\|u(s)\|_H^p \dint s.
        \end{split}
    \end{equation}
    Estimates~\eqref{ItoIneq} and \eqref{apriori} will be used several times to derive new estimates which ultimately lead to \eqref{eq:estapriorimain}.

    \smallskip
    \textit{Step 1: Estimating the supremum term}  $\E\sup\limits_{t\in[0, T]}\|u(s)\|_H^p$.

    Taking suprema and expectations in \eqref{ItoIneq}, we obtain the following estimate
    \begin{equation}\label{SPDEest1}
        \begin{split}
            \E\sup\limits_{r\in[0, t]}\|u(r\wedge\tau_n)\|_H^p &\leq \E\|u(0)\|_H^p + p\E\sup\limits_{r\in[0, t]}\int_0^{r\wedge\tau_n} \|u(s)\|_H^{p-2} B(s, u(s))^* u(s)\dint W(s)\\
            &\quad + \frac{p}{2}\E\int_0^{t\wedge\tau_n}\|u(s)\|_H^{p-2} f(s) \dint s+\frac{pK_c}{2}\E\int_0^{t\wedge\tau_n}\|u(s)\|_H^{p}\dint s.
        \end{split}
    \end{equation}
    Let $\varepsilon_1 > 0$. Then
    \begin{align*}
         & \E\sup\limits_{r\in[0, t]}\int_0^{r\wedge\tau_n}
        \|u(s)\|_H^{p-2}B(s, u(s))^*u(s)\dint W(s)                                                                                                                                                                                                  \\
         & \stackrel{\mathrm{(i)}}{\leq}2\sqrt{2}\E\Big(\int_0^{t\wedge\tau_n} \|u(s)\|_H^{2p-2}\|B(s, u(s))\|_{\calL_2(U, H)}^2 \dint s\Big)^{\frac{1}{2}}                                                                                        \\
         & \stackrel{\mathrm{(ii)}}{\leq} 2\sqrt{2}\Big(\E\sup\limits_{r\in[0, t]}\|u(r\wedge\tau_n)\|_H^p\Big)^{\frac{1}{2}}\Big(\E\int_0^{t\wedge \tau_n}\|u(s)\|_H^{p-2}\|B(s, u(s))\|_{\calL_2(U, H)}^2\dint s \Big)^{\frac{1}{2}}            \\
         & \stackrel{\mathrm{(iii)}}{\leq} \sqrt{2}\varepsilon_1\E\sup\limits_{r\in[0, t]}\|u(r\wedge\tau_n)\|_H^p + \frac{\sqrt{2}}{\varepsilon_1}\E\int_0^{t\wedge \tau_n}\|u(s)\|_H^{p-2}\|B(s, u(s))\|_{\calL_2(U, H)}^2 \dint s              \\
         & \stackrel{\mathrm{(iv)}}{\leq}\sqrt{2}\varepsilon_1\E\sup\limits_{r\in[0, t]}\|u(r\wedge\tau_n)\|_H^p\\
         &\quad  +\frac{\sqrt{2}}{\varepsilon_1}\E\int_0^{t\wedge \tau_n} \|u(s)\|_H^{p-2}(f(s)+K_B\|u(s)\|_H^2+K_\alpha\|u(s)\|_V^\alpha)\dint s,
    \end{align*}
    where in (i) we have applied the Burkholder-Davis-Gundy inequality with constant $2\sqrt{2}$ (see \cite[Theorem~1]{Ren08}), in (ii) H\"older's inequality, in (iii) Young's inequality and in (iv) hypothesis \condref{it:bound2}. Using the latter estimate in \eqref{SPDEest1}, we find
    \begin{align}\label{inter_1}
         & (1-p\sqrt{2}\varepsilon_1)\E\sup\limits_{r\in[0, t]}\|u(r\wedge\tau_n)\|_H^p \nonumber                                          \\
         & \leq \E\|u(0)\|_H^p + p\Big(\tfrac{\sqrt{2}}{\varepsilon_1}+\tfrac{1}{2}\Big)\E\int_0^{t\wedge \tau_n}\|u(s)\|_H^{p-2}f(s) \dint s \\ \nonumber & \quad +p\big(\tfrac{\sqrt{2}K_B}{\varepsilon_1}+\tfrac{K_c}{2}\big)\E\int_0^{t\wedge \tau_n}\|u(s)\|_H^p \dint s + \tfrac{pK_\alpha\sqrt{2}}{\varepsilon_1}\E\int_0^{t\wedge \tau_n}\|u(s)\|_H^{p-2}\|u(s)\|_V^\alpha \dint s.
    \end{align}
    Using estimate~\eqref{apriori} for the last term of \eqref{inter_1} leads to
    \begin{align}\nonumber
         & (1-p\sqrt{2}\varepsilon_1)\E\sup\limits_{r\in[0, t]}\|u(r\wedge\tau_n)\|_H^p                                                                                                                                                                            \\
         & \leq \big(1+K_\alpha\tfrac{2\sqrt{2}}{\varepsilon_1\theta}\big)\E\|u(0)\|_H^p + p\big(\tfrac{\sqrt{2}}{\varepsilon_1}+\tfrac{1}{2}+K_\alpha\tfrac{\sqrt{2}}{\varepsilon_1\theta}\big)\E\int_0^{t\wedge \tau_n}\|u(s)\|_H^{p-2}f(s) \dint s \label{inter_2} \\
         & \quad+ p\big(\tfrac{\sqrt{2}K_B}{\varepsilon_1}+\tfrac{K_c}{2}+\tfrac{\sqrt{2}K_\alpha K_c}{\varepsilon_1\theta}\big)\E\int_0^{t\wedge \tau_n}\|u(s)\|_H^p \dint s.\nonumber
    \end{align}
    It remains to absorb the integrals of $u$ on the right-hand side of \eqref{inter_2}.
    To this end, let $\varepsilon_2 > 0$. By H\"{o}lder's inequality and Young's inequality we obtain
    \begin{equation}\label{f_s_estimate}
        \begin{split}
            \E\int_0^{t\wedge \tau_n} \|u(s)\|_H^{p-2}f(s) \dint s &\leq \E\sup\limits_{r\in[0, t]}\|u(r\wedge\tau_n)\|_H^{p-2}\int_0^{t}f(s) \dint s\\
            &\leq\Big(\varepsilon_2\E\sup\limits_{r\in[0,t]}\|u(r)\|_H^p\Big)^{\frac{p-2}{p}} \Big(\varepsilon_2^{\frac{2-p}{2}}\E\Big(\int_0^{t} f(s) \dint s\Big)^{\frac{p}{2}}\Big)^{\frac{2}{p}}\\
            &\leq\tfrac{p-2}{p}\varepsilon_2\E\sup\limits_{r\in[0,t]}\|u(r\wedge\tau_n)\|_H^p+\tfrac{2}{p} \varepsilon_2^{\frac{2-p}{2}}\E\Big(\int_0^{t}f(s) \dint s\Big)^{\frac{p}{2}}.\\
        \end{split}
    \end{equation}
    Setting
    $\phi(\varepsilon_1, \varepsilon_2) = p\sqrt{2}\varepsilon_1  +(p-2)\varepsilon_2\big(\tfrac{\sqrt{2}}{\varepsilon_1}+\tfrac{1}{2}+K_\alpha\tfrac{\sqrt{2}}{\varepsilon_1\theta}\big)$
    and using \eqref{f_s_estimate} in \eqref{inter_2}
    we obtain:
    \begin{equation}\label{final_inter_estimate}
        \begin{split}
            (1-\phi(\varepsilon_1,\varepsilon_2))\E\sup\limits_{r\in[0, t]}\|u(r\wedge\tau_n)&\|_H^p  \leq \big(1+K_\alpha\tfrac{2\sqrt{2}}{\varepsilon_1\theta}\big)\E\|u(0)\|_H^p                                                                                                                                       \\
            &  +  2 \varepsilon_2^{\frac{2-p}{p}}\big(\tfrac{\sqrt{2}}{\varepsilon_1}+\tfrac{1}{2}+K_\alpha\tfrac{\sqrt{2}}{\varepsilon_1\theta}\big)\E\Big(\int_0^{t}f(s) \dint s\Big)^{\frac{p}{2}} \\
            &  + p\big(\tfrac{\sqrt{2}K_B}{\varepsilon_1}+\tfrac{K_c}{2}+\tfrac{\sqrt{2}K_\alpha K_c}{\varepsilon_1\theta}\big)\E\int_0^{t\wedge \tau_n}\|u(s)\|_H^p \dint s
        \end{split}
    \end{equation}
    First choosing $\varepsilon_1$ small enough, and then $\varepsilon_2$ such that $\phi(\varepsilon_1,\varepsilon_2)=\frac12$, it follows that there is a constant $C>0$ (only depending on $\alpha$, $\beta$, $\theta$, $p$, $K_c$, $K_A$, $K_B$, $K_{\alpha}$) such that
    \begin{equation}\label{intermediateresult}
        \E\sup\limits_{r\in[0, t]}\|u(r\wedge\tau_n)\|_H^p \leq C\Big(\E\|u(0)\|_H^p +\E\Big(\int_0^{t}f(s) \dint s\Big)^{\frac{p}{2}}+\E\int_0^{t}\one_{[0,\tau_n]}(s)\|u(s)\|_H^p\Big).
    \end{equation}
    Applying Gronwall's inequality to $v(t):= \sup_{r\in[0, t]}\|u(r\wedge \tau_n)\|_H^p$ we find

    \begin{equation*}
        \E\sup\limits_{t\in[0, T]}\|u(t\wedge\tau_n)\|_H^p \leq Ce^{CT}\Big(\E\|u(0)\|_H^p + \E\Big(\int_0^{T} f(s) \dint s\Big)^{\frac{p}{2}}\Big)
    \end{equation*}
    By Fatou's lemma this leads to
    \begin{equation}\label{sup_estimate}
        \E\sup\limits_{t\in[0, T]}\|u(t)\|_H^p \leq Ce^{CT}\Big(\E\|u(0)\|_H^p + \E\biggl(\int_0^{T}f(s) \dint s\biggr)^{\frac{p}{2}}\Big)
    \end{equation}
    and completes the proof of the supremum estimate.

    \smallskip

    \textit{Step 2: Estimating the $V$-norm} $\E\Big(\int_0^T \|u(s)\|_V^\alpha \dint s\Big)^{\frac{p}{2}}$.
    
    In order to estimate this quantity, by Lemma~\ref{ito_p} we find
    \begin{equation*}
	\begin{split}
        	\|u(t)\|_H^2 = \|u(0)\|_H^2 &+ \int_0^t \Big( 2\langle A(s, u(s)), u(s)\rangle + \|B(s, u(s))\|_{\calL_2(U, H)}^2 \Big) \dint s\\
						& \quad + 2\int_0^t B(s, u(s))^*u(s) \dint W(s)
	\end{split}
    \end{equation*}
    By the coercivity condition \condref{it:coerc} we find that
    \begin{align*}
        \|u(t)\|_H^2 & + \int_0^t \Big( (p-2)\frac{\|B(s, u(s))^*u(s)\|_U^2}{\|u(s)\|_H^2} + \theta\|u(s)\|_V^\alpha \Big) \dint s    \\
                    & \leq \|u(0)\|_H^2 +  \int_0^t \big(f(s) + K_c \|u(s)\|_H^2\big) \dint s + 2\int_0^t B(s, u(s))^* u(s) \dint W(s)
    \end{align*}
    Selecting just the term $\theta \|u(s)\|_V^\alpha$ and evaluating at $t= \tau_n$ gives
    \begin{equation}\label{eq:Vtermstep}
        \theta\int_0^{\tau_n} \|u(s)\|_V^\alpha \dint s \leq \|u(0)\|_H^2 + \int_0^t \big(f(s) + K_c \|u(s)\|_H^2\big) \dint s + 2\int_0^{\tau_n} B(s, u(s))^* u(s)\dint W(s)
    \end{equation}
    Applying the function $|\cdot |^{\frac{p}{2}}$ to both sides of \eqref{eq:Vtermstep} and taking expectations, we obtain
    \begin{equation}\label{V_norm}
        \begin{split}
            &\frac{\theta^{\frac{p}{2}}}{a_p}\E\Big(\int_0^{\tau_n}\|u(s)\|_V^\alpha \dint s\Big)^{\frac{p}{2}} \leq \E\|u(0)\|_H^p +  \fterm
            \\ & \qquad
            + K_c^{\frac p 2}  \E\Big(\int_0^{T} \|u(s)\|_H^2 \dint s\Big)^{\frac{p}{2}}
            + 2^{\frac{p}{2}} \E\left|\int_0^{\tau_n} B(s, u(s))^* u(s)\dint W(s)\right|^{\frac{p}{2}},
        \end{split}
    \end{equation}
    where $a_p = 2^{p-2}$.
    The $\|u(s)\|_H^2$-terms can be estimated with help of \eqref{sup_estimate} by
    \begin{equation}\label{est_ush_term}
        \begin{split}
            \E\Big(\int_0^{T} \|u(s)\|_H^2 \dint s\Big)^{\frac{p}{2}} &\le T^{\frac p 2} \E\sup\limits_{t\in[0, T]}\|u(t)\|_H^p\\
            &\leq C T^{\frac p 2} e^{CT}\Big(\E\|u(0)\|_H^p + \E\biggl(\int_0^{T}f(s) \dint s\biggr)^{\frac{p}{2}}\Big).
        \end{split}
    \end{equation}
    Thus it remains to estimate the $B$-term. We obtain:
    \begin{equation}\label{BDG_V_norm}
        \begin{split}
            \E\Big|\int_0^{\tau_n} &B(s, u(s))^* u(s)\dint W(s)\Big|^{\frac{p}{2}}
            \stackrel{\mathrm{(i)}}{\leq} C_p\E\Big(\int_0^T \|u(s)\|_H^2 \one_{[0,\tau_n]}(s) \|B(s, u(s))\|_{\calL_2(U, H)}^2 \dint s\Big)^{\frac{p}{4}}\\
            &\stackrel{\mathrm{(ii)}}{\leq} C_p\E\Big(\sup\limits_{t\in[0, T]}\|u(t)\|_H^2 \int_0^{\tau_n}  \|B(s, u(s))\|_{\calL_2(U, H)}^2 \dint s\Big)^{\frac{p}{4}}\\
            &\stackrel{\mathrm{(ii)}}{\leq} C_p\Big(\E\sup\limits_{t\in[0, T]}\|u(t)\|_H^p\Big)^{\frac{1}{2}}\Big(\E\Big(\int_0^{\tau_n} \|B(s, u(s))\|_{\calL_2(U, H)}^2 \dint s\Big)^{\frac{p}{2}}\Big)^{\frac{1}{2}}\\
            &\stackrel{\mathrm{(iii)}}{\leq} C_p\frac{1}{2\varepsilon}\E\sup\limits_{t\in[0, T]}\|u(t)\|_H^p + C_p\frac{\varepsilon}{2}\E\Big(\int_0^{\tau_n} \|B(s, u(s))\|_{\calL_2(U, H)}^2 \dint s\Big)^{\frac{p}{2}},
        \end{split}
    \end{equation}
    where in (i) we have applied the Burkholder-Davis-Gundy inequality, (ii) follows from H\"older's inequality, and (iii) is a consequence of Young's inequality. Applying \condref{it:bound2}, the $B$-term can be estimated as
    \begin{align*}
         & \E\Big(\int_0^{\tau_n} \|B(s, u(s))\|_{\calL_2(U, H)}^2 \dint s\Big)^{\frac{p}{2}} \leq \E\Big(\int_0^{\tau_n} \big( f(s) + K_B \|u(s)\|_H^2 + K_{\alpha} \|u(s)\|_V^\alpha \big) \dint s\Big)^{\frac{p}{2}}          \\
         & \leq b_p\E\Big(\int_0^T f(s) \dint s \Big)^{\frac{p}{2}} + b_p K_B^{\frac p 2} \E\Big(\int_0^T \|u(s)\|_H^2 \dint s\Big)^{\frac{p}{2}}+ b_pK_\alpha^{\frac{p}{2}}\E\Big(\int_0^{\tau_n} \|u(s)\|_V^\alpha \dint s\Big)^{\frac{p}{2}} \\
         & \leq b_p\fterm + b_p K_B^{\frac p 2} T^{\frac{p}{2}}\E\sup\limits_{t\in[0, T]}\|u(t)\|_H^p + b_p K_\alpha^{\frac{p}{2}}\E\Big(\int_0^{\tau_n} \|u(s)\|_V^\alpha \dint s \Big)^{\frac{p}{2}},
    \end{align*}
    where $b_p = 3^{\frac{p-2}{2}}$.
    Recombining this estimate with \eqref{sup_estimate} and \eqref{BDG_V_norm}, we obtain:
    \begin{align*}
        \E & \Big|\int_0^t B(s, u(s))^*  u(s) \dint W(s)\Big|^{\frac{p}{2}} \leq \Big(C_p \frac{1}{2\varepsilon} + b_pK_B^{\frac p 2} C_pT^{\frac{p}{2}}\frac{\varepsilon}{2} \Big) \E\sup\limits_{t\in[0, T]}\|u(t)\|_H^p             \\ &\qquad \qquad + b_pC_p\frac{\varepsilon}{2}\fterm + b_pK_\alpha^{\frac{p}{2}} C_p\frac{\varepsilon}{2}\E\Big(\int_0^{\tau_n} \|u(s)\|_V^\alpha \dint s\Big)^{\frac{p}{2}}\\
           & \leq C_\varepsilon(1+T^{\frac{p}{2}}) e^{CT}\Big[\E\|u(0)\|_H^p +  \fterm\Big] + b_p K_\alpha^{\frac{p}{2}} C_p\frac{\varepsilon}{2}\E\Big(\int_0^{\tau_n} \|u(s)\|_V^\alpha \dint s\Big)^{\frac{p}{2}}
    \end{align*}
    Using this and \eqref{est_ush_term} in \eqref{V_norm}, it follows that
    \begin{align*}
        \theta^{\frac{p}{2}}\E\Big(\int_0^{\tau_n}\|u(s)\|_V^\alpha \dint s\Big)^{\frac{p}{2}} & \leq C_\varepsilon'(1+T^{\frac{p}{2}}) e^{C T}\Big[\E\|u(0)\|_H^p + \fterm\Big]                                                                     \\
                                                                                              & \phantom{\leq} + a_p b_p 2^{\frac{p-2}{2}}K_\alpha^{\frac{p}{2}} C_p\varepsilon\E\Big(\int_0^{\tau_n} \|u(s)\|_V^\alpha \dint s\Big)^{\frac{p}{2}}.
    \end{align*}
    Therefore, choosing $\varepsilon > 0$ small enough, we obtain
    \begin{equation*}
        \E\Big(\int_0^T \|u(s)\|_V^\alpha \dint s\Big)^{\frac{p}{2}} \leq C''e^{C''T}\Big(\E\|u(0)\|_H^p + \fterm\Big).
    \end{equation*}

    Since we have estimated all three terms in the above steps, this finishes the proof.
\end{proof}

\begin{remark}
    One can also prove an estimate for the integral of $\|u(s)\|_H^{p-2}\|u(s)\|_V^\alpha$. Indeed, by H\"older's and Young's inequality
    \begin{align*}
        \E\int_0^{T} \|u(s)\|_H^{p-2} \|u(s)\|_V^\alpha \dint s & \leq \E \sup_{s\in [0,T]}\|u(s)\|_H^{p-2} \int_0^{T} \|u(s)\|_V^\alpha \dint s
        \\
        & \le \tfrac{p-2}{p} \E \sup_{t \in [0,T]} \|u(t)\|_H^p + \tfrac 2 p \E \Big( \int_0^T \|u(t)\|_V^\alpha \dint t\Big)^{\frac p 2},
    \end{align*}
    where the last line is bounded by the left-hand side of \eqref{eq:estapriorimain}.
\end{remark}

If $K_B = K_c = 0$ in Assumptions~\ref{main_assumptions} \condref{it:coerc} and \condref{it:bound2}, it is possible to improve the dependency on $p$ in estimate~\eqref{eq:estapriorimain}. Here the condition $p\geq \beta+2$ is not needed.
\begin{corollary}\label{cor:a_priori_remark}
    Suppose $u$ is a solution of equation \eqref{eq:variationalintro} with initial condition $u(0)\in L^{p}(\Omega; H)$ and \condref{it:coerc}, \condref{it:bound1}, \condref{it:bound2} from Assumptions ~\ref{main_assumptions} hold with $K_B = K_c=0$ and $f\in L^\frac{p}{2}(\Omega; L^1([0, T]))$. Then there exists a constant $C$ only depending on $\alpha, \beta, \theta, K_A, K_{\alpha}$ such that
    \begin{equation}\label{eq:estimateforallpsup}
    \begin{aligned}
        \|u\|_{L^p(\Omega; C([0, T]; H))} + & p^{-1/2}\|u\|_{L^p(\Omega;L^2([0,T];V))} \\ & \leq C\big[\|u(0)\|_{L^p(\Omega; H)} + \|f\|^{\frac{1}{2}}_{L^p(\Omega; L^1(0, T))}\big].
    \end{aligned}
    \end{equation}
    Moreover, if $B(v)^*v = 0$ for all $v\in V$, then
    the above estimates hold for all $p\in [2, \infty]$, and $p^{-1/2}$ can be omitted.
\end{corollary}
The main point is that $C$ does not depend on $p$ and $T$. In particular, we can let $T\to\infty$ in \eqref{eq:estimateforallpsup} if $f$ is integrable over $\R_+$.

\begin{proof}
    Estimate~\eqref{final_inter_estimate} gives for every $\varepsilon_1, \varepsilon_2 > 0$:
    \begin{equation}\label{p_indep_start}
        \begin{split}
            & \big(1-\phi(\varepsilon_1,\varepsilon_2)\big)\E\sup\limits_{t\in[0, S]}\|u(t\wedge\tau_n)\|_H^p   \\
            & \qquad \leq \big(1+K_\alpha\tfrac{2\sqrt{2}}{\varepsilon_1\theta}\big)\E\|u(0)\|_H^p + 2\varepsilon_2^{\frac{2-p}{p}}\big(\tfrac{\sqrt{2}}{\varepsilon_1}+\tfrac{1}{2}+K_\alpha\tfrac{\sqrt{2}}{\varepsilon_1\theta}\big)\E\Big(\int_0^{\tau_n}f(s) \dint s\Big)^{\frac{p}{2}},\\
        \end{split}
    \end{equation}
    where $\phi(\varepsilon_1, \varepsilon_2) = p\sqrt{2}\varepsilon_1  +(p-2)\varepsilon_2\big(\tfrac{\sqrt{2}}{\varepsilon_1}+\tfrac{1}{2}+K_\alpha\tfrac{\sqrt{2}}{\varepsilon_1\theta}\big)$. Choosing
    $$\varepsilon_1 = \frac{1}{2\sqrt{2}p}, \qquad \varepsilon_2 = \frac{1}{2(p-2)(8p+1+K_\alpha\frac{8p}{\theta})}$$
    gives $\phi(\varepsilon_1,\varepsilon_2) = \frac 1 4$. Moreover,
    $$\frac{1}{\varepsilon_2} \leq 16p^2(1+K_\alpha \tfrac{1}{\theta}) + p^2 + 1 \leq A p^2,$$
    where $A$ is a constant depending on $K_\alpha$ and $\theta$. Therefore, we get:
    \begin{align*}
        \tfrac{1}{4}\E\sup\limits_{t\in[0, S]}\|u(t\wedge\tau_n)\|_H^p & \leq (1+K_\alpha\tfrac{8p}{\theta})\E\|u(0)\|_H^p + (Ap^2+1)^{\frac{p-2}{p}}\E\Big(\int_0^{\tau_n} f(s) \dint s\Big)^{\frac{p}{2}}
    \end{align*}
    Taking $1/p$-th powers, the supremum of \eqref{eq:estimateforallpsup} follows since for every $\gamma>0$,
    \[
        \sup_{p\in [2, \infty)}p^{\gamma/p} = \sup_{p \in [2,\infty)} \big(1+(p-1)\big)^{\gamma/p} \le \sup_{p \in [2,\infty)} e^{\gamma (p-1)/p} = e^\gamma.
    \]
    Under the additional assumption $B(v)^*v=0$, it follows that condition~\condref{it:coerc} holds for all $p\in [2, \infty)$. Therefore, we can let $p\to \infty$ in \eqref{eq:estimateforallpsup}.

    In order to derive the estimate~\eqref{eq:estimateforallpsup} for the $V$-term, we use \eqref{V_norm} and the assumption $K_c = 0$ to find that
    \begin{equation}\label{V_norm_ineq}
        \begin{split}
            \frac{\theta^{\frac{p}{2}}}{a_p}\E\Big(\int_0^{T}  \|u(s)\|_V^\alpha &\dint s\Big)^{\frac{p}{2}}
            \leq \E\|u(0)\|_H^p + \fterm  \\
            & + 2^{\frac{p}{2}}\E\Big|\int_0^{T} B(s, u(s))^*u(s)\dint W(s)\Big|^{\frac{p}{2}} ,
        \end{split}
    \end{equation}
    where $a_p = 2^{p-2}$. If $B(v)^*v = 0$ for all $v \in V$, then the stochastic integral vanishes and thus \eqref{V_norm_ineq} already implies the required result.

    It remains to prove estimate~\eqref{eq:estimateforallpsup} for the $V$-norm in the case the stochastic integral in \eqref{V_norm_ineq} does not vanish. For this we use the Burkholder-Davis-Gundy inequality with $\gamma_p = \frac{(2p)^{p/4}}{2}$  as in \eqref{BDG_V_norm} (see \cite[Theorem~A]{carlen_1991}), giving for all $\varepsilon>0$
    \begin{align*}
        &\E\Big|\int_0^{T}  B(s, u(s))^*u(s)\dint W(s)\Big|^{\frac{p}{2}}\\
        &\leq \frac{\gamma_p }{\varepsilon}\E\sup\limits_{t\in[0, T]}\|u(t)\|_H^p + \gamma_p \varepsilon \E\Big(\int_0^{T} \|B(s, u(s))\|_{\calL_2(U, H)}^2 \dint s \Big)^{\frac{p}{2}},
    \end{align*}
    where $\varepsilon > 0$ is arbitrary. Using assumption \condref{it:bound2}, we additionally obtain
    \begin{equation*}
        \begin{split}
            &\E\Big(\int_0^{T} \|B(s, u(s))\|_{\calL_2(U, H)}^2 \dint s \Big)^{\frac{p}{2}}\\
            &+\leq 2^{\frac{p-2}{2}}\fterm + 2^{\frac{p-2}{2}}K_\alpha^{\frac{p}{2}}\E\Big(\int_0^{T} \|u(s)\|_V^\alpha \dint s\Big)^{\frac{p}{2}}.
        \end{split}
    \end{equation*}
    Recombining all terms with inequality \eqref{V_norm_ineq} we find
    \begin{align*}
        \frac{\theta^{\frac{p}{2}}}{a_p} \E\Big(\int_0^{T} & \|u(s)\|_V^\alpha \dint s\Big)^{\frac{p}{2}} \leq \E\|u(0)\|_H^p + (1+ \gamma_p \varepsilon 2^{p-1}) \fterm
        \\ &+
        \tfrac{2^{\frac{p}{2}}\gamma_p}{\varepsilon}\E\sup\limits_{t\in[0, T]} \|u(t)\|_H^p
        + 2^{p-1} K_\alpha^{\frac{p}{2}}\gamma_p \varepsilon \E\Big(\int_0^{T}\|u(s)\|_V^\alpha\Big)^{\frac{p}{2}}.
    \end{align*}
    Therefore, setting $\varepsilon = \frac{\theta^{\frac{p}{2}}}{a_p 2^{p+1} K_\alpha^{\frac{p}{2}}\gamma_p}$
    we obtain
    \begin{align*}
        \frac{\theta^{\frac{p}{2}}}{2 a_p} \E\Big(\int_0^{T}  \|u(s)\|_V^\alpha \dint s\Big)^{\frac{p}{2}} & \leq \E\|u(0)\|_H^p + (1+ \gamma_p \varepsilon 2^{p-1}) \fterm
        \\ & \qquad +
        \tfrac{2^{\frac{p}{2}}\gamma_p}{\varepsilon}\E\sup\limits_{t\in[0, T]} \|u(t)\|_H^p.
    \end{align*}
    Taking $p$-th powers and observing that the leading term is $\gamma_p^{2/p}\leq C \sqrt{p}$, we arrive at the desired inequality.
\end{proof}

Given the a priori estimates of Theorem~\ref{a_priori_theorem}, one can now complete the proof of Theorem~\ref{Main_theorem} by showing existence and uniqueness as in the classical case $p=2$. Details are standard and can be found in \cite{rockner_2010}. As our assumptions differ from the latter some changes are required, and in particular, we require $p\geq \beta+2$, which is needed for technical reasons in the existence proof, but can often be avoided by a localization argument. Note that it was not used in Theorem~\ref{a_priori_theorem}. For details we refer to the existence and uniqueness proofs in \cite{brzezniak_strong_2014, neelima_2020}.

\section{Applications}\label{sec:appl}
In this section, we apply our framework to
\begin{itemize}
    \item linear scalar second-order parabolic equations, namely the stochastic heat equation with both Dirichlet (section~\ref{stoch_heat_dir}) and Neumann boundary conditions (section~\ref{stoch_heat_neu}), in which the $p$-dependent term in the coercivity condition \ref{it:coerc} reduces to the classical setting in certain cases.
    \item semilinear second-order parabolic equations, namely the stochastic Burgers' equation (section~\ref{stoch_burg}) and the stochastic Navier-Stokes equations in two dimensions (section~\ref{stoch_nav_sto}),
    \item systems of SPDEs (section~\ref{systems_SPDE}) and higher-order SPDEs (section~\ref{higher_SPDE}) as treated in \cite{du_2020,wang_2020},
    \item the fully nonlinear evolution induced by the $p$-Laplacian influenced by noise (section~\ref{stoch_p_lapl}).
\end{itemize}
The treated examples demonstrate the wide range of applicability of our unifying abstract framework. In several cases the regularity estimates in $L^p(\Omega)$ for $p> 2$ seem new. In all cases, the approach to prove them via our Theorem \ref{Main_theorem} also seems new. The variety of the examples will hopefully be enough to explain the reader how to apply our framework to concrete SPDEs.

\subsection{Stochastic heat equation with Dirichlet boundary conditions}\label{stoch_heat_dir}
We consider a stochastic heat equation with additive noise and Dirichlet boundary conditions.
\begin{equation}\label{eq: stoch_heat_equation_dir}
    \dint u(t) = \Big(\sum\limits_{i, j =1}^d \partial_i (a^{ij}\partial_{j} u(t)) + \phi(t) \Big)\dint t + \sum\limits_{k=1}^\infty\Big(\sum\limits_{i=1}^d b_k^i\partial_i u(t) + \psi_{k, t} \Big) \dint W_k(t).
\end{equation}

Here the $W_k(t)$ are real-valued Wiener processes. In what follows, we use:
\begin{assumptions}\label{ass: stoch_heat_equation_dir}
    Let $\mathcal{D} \subseteq \R^d$ be a an open set. Let
    $$(V, H, V^*) = (H_0^1(\mathcal{D}), L^2(\Distr), H^{-1}(\Distr))$$ and $U = \ell^2$.
Suppose that $a^{ij} \in L^\infty(\Omega \times [0,T] \times \mathcal D)$ for $1 \le i,j \le d$ and $(b_k^i)_{k = 1}^\infty \in L^\infty(\Omega \times [0,T]; W^{1,\infty}(\mathcal D;\ell^2))$ for $1\leq i \leq d$. Furthermore, we assume that the coefficients are progressively measurable. Define
    \begin{equation}\label{eq: sigma_ij_dir}
        \sigma^{ij} = \sum\limits_{k=1}^\infty b_k^{i} b_k^{j}, \qquad i, j \in\mathbb{N}
    \end{equation}
    and suppose that the uniform ellipticity condition on $a^{ij}$ and $b_k^{i}$:
    \begin{equation}\label{cond:ellipticity_condition_dir}
        \sum\limits_{i, j = 1}^d \left(2a^{ij}-\sigma^{ij}\right)\xi^i\xi^j \geq \theta |\xi|^2 \qquad \text{for all } \xi \in \mathbb{R}^d
    \end{equation}
    holds true where $\theta > 0$. Furthermore, assume $\phi \in L^{p}(\Omega; L^2([0, T]; H^{-1}(\Distr)))$,\\ $\psi \in L^{p}(\Omega; L^2([0, T]; L^2(\Distr;\ell^2)))$, and $u_0 \in L^{p}(\Omega; L^2(\mathcal{D}))$, where $p \geq 2$.
\end{assumptions}
Equation~\eqref{eq: stoch_heat_equation_dir} can be reformulated as a stochastic evolution equation of the form $$\dint u(t) = A(t, u(t)) \ \dint t + \sum\limits_{k=1}^\infty B_k(t, u(t)) \ \dint W_k(t),$$
with the deterministic linear operator $A(t): H^1_0(\mathcal{D}) \to H^{-1}(\mathcal{D})$ defined by
\begin{equation*}
    \langle A(t, u), v\rangle = -\sum_{i,j=1}^{d}\int_\mathcal{D} a^{ij}(\partial_i u \, \partial_j v \, \dint x +\langle \phi(t), v\rangle \qquad \text{for } u, v\in H_0^1(\mathcal{D}),
\end{equation*}
and stochastic operators $B_k(t): H_0^1(\Distr) \to L^2(\Distr)$ given by
\begin{equation*}
    B_k(t, v) = \sum\limits_{i=1}^d b_k^{i} \partial_i v + \psi_{k, t} \qquad \text{for } v\in H_0^1(\mathcal{D}).
\end{equation*}
It turns out that the $p$-dependent term in the coercivity condition \condref{it:coerc} vanishes. Therefore, the solution admits moment estimates of all orders $p \geq 2$, only limited by the integrability of the additive noise and the initial condition:

\begin{proposition}\label{prop: stoch_heat_equation_dir}
    Suppose that Assumptions~\ref{ass: stoch_heat_equation_dir} are satisfied. Then, a unique variational solution $u$ of equation~\eqref{eq: stoch_heat_equation_dir} in the sense of Definition~\ref{solution_definition} exists and the following estimates hold:
    \begin{align*}
        \E&\sup\limits_{t\in[0, T]}\|u(t)\|_{L^2(\mathcal{D})}^p + \E\Big(\int_0^T \|u(t)\|_{H_0^1(\Distr)}^2 \dint t\Big)^{\frac{p}{2}}
        \\ & \leq Ce^{CT}\bigg(\E\|u(0)\|_{L^2(\mathcal{D})}^p + \E \Big(\int_0^T \|\phi(t)\|_{H^{-1}(\mathcal{D})}^2 \dint t\Big)^{\frac{p}{2}}  + \E \Big(\int_0^T \|\psi(t)\|_{L^2(\mathcal{D};\ell^2)}^2 \dint t\Big)^{\frac{p}{2}}\bigg)
    \end{align*}
    where $C$ depends on $\theta, p, a^{ij}$ and $b_k^i$ for all $i, j, k \in \mathbb{N}$.
\end{proposition}
\begin{remark}
    Assuming that $\Distr$ is bounded and all $b_k$ are not space dependent, we can use Corollary~\ref{cor:a_priori_remark} to obtain $p$-independent constants, and even take $p= \infty$. That is, there exists a constant $C$ such that for all $p\in [2, \infty]$
    \begin{align*}
        &\|u\|_{L^p(\Omega; C([0, T]; L^2(\Distr)))} + \|u\|_{L^p(\Omega;L^2(0,T;H^1_0(\Distr)))} \\ & \leq C \Big[\|u(0)\|_{L^p(\Omega; L^2(\Distr))}  + \|\phi\|_{L^p(\Omega;L^2(0,T;H^{-1}(\mathcal D)))}                                                                       + \|\psi\|_{L^p(\Omega;L^2(0,T;L^2(\mathcal{D};\ell^2)))}\Big]
    \end{align*}
    where $C$ only depends on $\theta, a^{ij}, b_k^{i}$ for all $i, j, k, \in \mathbb{N}$.
\end{remark}
\begin{remark}\label{rem:Dirichletnoreg}
    A version of Proposition~\ref{prop: stoch_heat_equation_dir} holds if we only assume $(b_k^i)_{k = 1}^\infty \in L^\infty(\Omega \times [0,T]\times \mathcal D;\ell^2))$. However, in this case we can only use $\frac{\|B(t, v)^*v\|_U^2}{\|v\|_H^2}\leq \|B(t, v)\|^2_{\calL_2(U, H)}$ which leads to the $p$-dependent coercivity condition 
    \begin{equation*}
        \sum\limits_{i, j = 1}^d \left(2a^{ij}-(p-1)\sigma^{ij}\right)\xi^i\xi^j \geq \theta |\xi|^2 \qquad \text{for all } \xi \in \mathbb{R}^d.
    \end{equation*}
\end{remark}

\begin{proof}[Proof of Proposition~\ref{prop: stoch_heat_equation_dir}]
    By Remark~\ref{rem:additive} and Theorem~\ref{Main_theorem}, it suffices to show Assumptions~\ref{main_assumptions}, \condref{it:hem}-\condref{it:bound2}, for $(A,B)$ with $\alpha = 2$, $\phi = 0$, $\psi = 0$, and $f = 0$.
    Hemicontuinty \condref{it:hem} is immediate from the definition of $A$. For local weak monotonicity \condref{it:weak_mon}, observe that it suffices to prove the inequality for $v \in H_0^1(\mathcal{D})$ and $u = 0$ by linearity.
    Using uniform ellipticity \eqref{cond:ellipticity_condition_dir}, it follows:
    \begin{align}\label{eq: stoch_heat_equation_dir_H2}
        \begin{split}
            2\langle A(t, v), v\rangle + &
            \|B(t, v)\|^2_{\calL_2(\ell^2, L^2(\Distr))}\\
            &= -\sum\limits_{i, j=1}^\infty\int_\mathcal{D}2a^{ij} \partial_i v \, \partial_j v \, \dint x + \sum\limits_{k=1}^\infty \int_\mathcal{D} \sum\limits_{i, j = 1}^d b_k^{i} b_k^{j} \partial_i v \, \partial_j v \, \dint x\\
            &\ = \sum\limits_{i, j = 1}^d \int_{\mathcal{D}} (-2a^{ij}+\sigma^{ij}) \, \partial_i v \, \partial_j v \, \dint x\\
            &\leq-\theta\|v\|_{H_0^1(\mathcal{D})}^2 + \theta\|v\|_{L^2(\Distr)}^2,
        \end{split}
    \end{align}
    that is, \condref{it:weak_mon} is satisfied with $K =\theta$ (if $\Distr$ is bounded one can take $K=0$ by Poincar\'e's inequality). For coercivity \condref{it:coerc}, observe that the first two terms in \condref{it:coerc} form the first line of \eqref{eq: stoch_heat_equation_dir_H2}.  Therefore, it remains to derive an expression for $\|B(t, v)^*v\|_{\ell^2}^2/\|v\|_{L^2(\mathcal{D})}^2$, where $v\in H_0^1(\mathcal{D})$. Integration by parts gives

    \begin{equation*}
            (B(t, v)^*v)_k = \int_\mathcal{D} b^{i}_k \partial_i v v \dint x = \frac{1}{2}\int_\mathcal{D} \partial_i b^{i}_k v^2 \dint x.
    \end{equation*}
    Using the spatial regularity of $b_k^{i}$, we obtain:
    \begin{equation*}
        \begin{split}
            \Big \|k\mapsto \sum\limits_{i=1}^d \int_\mathcal{D} (b^i_k \partial_i v) v \dint x\Big \|_{\ell^2} &= \Big \|k\mapsto \sum\limits_{i=1}^d \int_\mathcal{D} \frac{1}{2}(\partial_i b^i_k) v^2 \dint x\Big \|_{\ell^2}\\
            &\leq \frac{1}{2} \int_\mathcal{D} \|\DIV(b)\|_{\ell^2} v^2 \dint x\\
            &\leq \|\DIV(b)\|_{L^\infty(\mathcal{D}; \ell^2)} \|v\|_{L^2(\mathcal{D})}^2.
        \end{split}
    \end{equation*}
    Therefore,
    \begin{equation*}
        \begin{split}
            2\langle A(t, v), v\rangle + \sum\limits_{k=1}^\infty \Big\|\sum\limits_{i=1}^d b^{ik}\partial_i v\Big\|_{L^2(\mathcal{D})}^2 &+ (p-2)\frac{\|(B_t(v)^*v)\|^2_{\ell^2}}{\|v\|_{L^2(\mathcal{D})}^2}\\
            &\leq -\theta\|v\|_{H_0^1(\mathcal{D})}^2 + C(p-2) \|v\|_{L^2(\mathcal{D})}^2,
        \end{split}
    \end{equation*}
    that is, \condref{it:coerc} is satisfied with $f = 0$ and $K_c = C(p-2)$.
    For the boundedness condition \condref{it:bound1}, let $u, v\in H_0^1(\mathcal{D})$. Then,
    \begin{equation*}
        \begin{split}
            |\langle A(t, u), v\rangle | \leq \sum\limits_{i,j = 1}^d \|a^{ij}\|_{L^\infty(\Distr)} \|u\|_{H_0^1(\mathcal{D})}\|v\|_{H_0^1(\mathcal{D})},
        \end{split}
    \end{equation*}
    that is, $\|A(t, u)\|_{H^{-1}(\mathcal{D})}^2 \leq \left(\sum\|a^{ij}\|_{L^\infty(\Omega \times [0,T] \times \mathcal D)}\right)^2 \|u\|_{H_0^1(\mathcal{D})}^2$, implying \condref{it:bound1} for $\alpha = 2$, $\beta = 0$, and $K_A = \left(\sum\|a^{ij}\|_{L^\infty(\Omega \times [0,T] \times \mathcal D)}\right)^2/2$. Similarly, because of \eqref{eq: sigma_ij_dir},
    \begin{equation*}
        \|B(t, v)\|_{L^2(\mathcal{D};\ell^2)}^2 \leq \Big\|\sum\limits_{i, j = 1}^d \sigma^{ij}\Big\|_{L^\infty(\Distr)} \|v\|_{H_0^1(\mathcal{D})}^2.
    \end{equation*}
    Hence condition \condref{it:bound2} holds with $K_\alpha = \Big\|\sum\limits_{i, j = 1}^d \sigma^{ij}\Big\|_{L^\infty(\Omega \times [0,T] \times \mathcal D)}$ and $K_B = 0$.
\end{proof}
From the above proof it follows that the regularity condition on $b$ in Assumption~\ref{ass: stoch_heat_equation_dir} can actually be weakened to $b\in L^\infty(\Omega \times [0,T]\times \mathcal D;\ell^2)$ and $\DIV(b)\in L^\infty(\mathcal{D}; \ell^2)$, and where the divergence only needs to exist in distributional sense.

\subsection{Stochastic heat equation with Neumann boundary conditions}\label{stoch_heat_neu}
The second equation we consider is the same stochastic heat equation as before, but now with Neumann boundary conditions on a domain $\mathcal{D} \subseteq \mathbb{R}^d$. For completeness, this equation is:
\begin{equation}\label{eq:stoch_heat_eq_neumann}
    \dint u(t) = \Big(\sum\limits_{i, j =1}^d \partial_i (a^{ij}\partial_{j} u(t)) + \phi(t) \Big)\dint t + \Big(\sum\limits_{k=1}^\infty\sum\limits_{i=1}^d b_k^{i}\partial_i u(t) + \psi_{t, k} \Big) \dint W_k(t).
\end{equation}
Most assumptions and computations will be similar as before, though some special care is needed to derive the coercivity condition in the Neumann setting.

\begin{assumptions}\label{ass: stoch_heat_eq_neumann}
    Let $p\in [2, \infty)$. Let $\mathcal{D} \subseteq \R^d$ be a bounded $C^1$-domain, and consider
    $$(V, H, V^*) = (H^1(\mathcal{D}), L^2(\Distr), H^1(\Distr)^*).$$
    Suppose that $a^{ij} \in L^\infty(\Omega \times [0,T] \times \mathcal D)$ for $1 \le i,j \le d$ and $(b_k^{i})_{k=1}^\infty \in L^\infty(\Omega \times [0,T]; W^{1,\infty}(\mathcal D;\ell^2))$ for $1\leq i \leq d$. Furthermore, we assume that the coefficients are progressively measurable. Define
    \begin{equation}\label{sigma_ij}
        \sigma^{ij} = \sum\limits_{k=1}^\infty b_k^{i} b_k^{j}, \qquad i, j \in\mathbb{N}
    \end{equation}
    and suppose that the uniform ellipticity condition on $a^{ij}$ and $b_k^{i}$:
    \begin{equation}\label{eq: stoch_heat_eq_neumann_ellipticity_condition}
        \sum\limits_{i, j = 1}^d \left(2a^{ij}-\sigma^{ij} - (p-2)C_b^2\right)\xi_i\xi_j \geq \theta |\xi|^2 \qquad \text{for all } \xi \in \mathbb{R}^d
    \end{equation}
    holds true where $\theta > 0$, and $C_b = \|b\cdot n\|_{L^\infty(\partial D; \ell^2)}$. Furthermore, assume $u_0 \in L^{p}(\Omega; L^2(\mathcal{D}))$,
    \[\phi \in L^{p}(\Omega; L^2([0, T]; H^1(\Distr)^*)) \ \ \text{and} \  \ \psi \in L^{p}(\Omega; L^2([0, T]; H^1(\Distr;\ell^2))).
    \]
\end{assumptions}

Equation~\eqref{eq:stoch_heat_eq_neumann} can be reformulated as a stochastic evolution equation of the form $$\dint u(t) = A(t, u(t)) \ \dint t + \sum\limits_{k=1}^\infty B_k(t, u(t)) \ \dint W_k(t),$$
with the deterministic linear operator $A(t): H^1(\mathcal{D}) \to H^1(\mathcal{D})^*$ defined by
\begin{equation}\label{eq: stoch_heat_equation_neumann_A}
    \langle A(t, u), v\rangle = -\sum_{i,j=1}^{d}\int_\mathcal{D} a^{ij}(t, x) \partial_i u \, \partial_j v \, \dint x +\langle \phi(t), v\rangle \qquad \text{for } u, v\in H^1(\mathcal{D}),
\end{equation}
and stochastic operators $B_k(t): H^1(\Distr) \to L^2(\Distr)$ given by
\begin{equation}\label{eq: stoch_heat_equation_neumann_B}
    B_k(t, v) = \sum\limits_{i=1}^d b_k^{i} \partial_i v + \psi_{k, t} \qquad \text{for } v\in H^1(\mathcal{D}).
\end{equation}
Unlike section~\ref{stoch_heat_dir}, the $p$-dependent term in the coercivity condition \condref{it:coerc} does not vanish completely and enters through the term $b\cdot n|_{\partial \Distr}$. If $b\cdot n$ vanishes at the boundary of $\Distr$, then the $p$-dependent term  solution admits moment estimates of all orders $p \geq 2$, only limited by the integrability of the additive noise and the initial condition. The main result for the Neumann case is:

\begin{proposition}\label{prop:Neumann}
    Suppose Assumptions~\ref{ass: stoch_heat_eq_neumann} hold. Then, a unique solution $u$ of equation \eqref{eq:stoch_heat_eq_neumann} exists and the following estimate holds:
    \begin{align*}
        & \E\sup\limits_{t\in[0, T]}\|u_t\|_{L^2(\mathcal{D})}^p + \E\Big(\int_0^T \|u_t\|_{H^1(\Distr)}^2 \dint t\Big)^{\frac{p}{2}} \\ & \leq Ce^{CT}\bigg(\E\|u_0\|_{L^2(\mathcal{D})}^p + \E \Big(\int_0^T \|\phi(t)\|_{H^{1}(\mathcal{D})^*}^2 \dint t\Big)^{\frac{p}{2}}  + \E \Big(\int_0^T \|\psi(t)\|_{L^2(\mathcal{D};\ell^2)}^2 \dint t\Big)^{\frac{p}{2}}\bigg)
    \end{align*}
    where $C$ depends on $\theta, p, a^{ij}$ and $b_k^{i}$ for all $i, j, k \in \mathbb{N}$.
\end{proposition}
Remark~\ref{rem:Dirichletnoreg} applies in the Neumann case as well, and thus this gives an alternative to \eqref{eq: stoch_heat_eq_neumann_ellipticity_condition} which additionally works without smoothness of $b$.

Before starting the proof of the proposition, we state a lemma that is needed to show the coercivity condition \condref{it:coerc}.
\begin{lemma}\label{lemma: neumann_coercivity}
    Consider assumptions ~\ref{ass: stoch_heat_eq_neumann} and $B$ as defined in \eqref{eq: stoch_heat_equation_neumann_B}. For every $\varepsilon \in (0, 1)$ there exists a constant $C_\varepsilon > 0$ such that for every nonzero $v \in H^1(\D)$ one has
    \begin{equation}
        \frac{\|B(t, v)^*v\|_{\ell^2}^2}{\|v\|_{L^2(\D)}^2} \leq (1+\varepsilon)C_b^2\|\nabla v\|_{L^2(\D)}^2 + C_\varepsilon(C_b^2 + D_b^2)\|v\|_{L^2(\D)}^2,
    \end{equation}
    where $C_b^2 = \|b \cdot n\|_{L^\infty(\partial D; \ell^2)}^2$ and $D_b^2 = \|\DIV(b)\|_{L^\infty(D; \ell^2)}^2$, where $n$ is the outer normal and $\DIV$ denotes the divergence.
\end{lemma}

\begin{proof}
    Observe that $\Tr(\phi u) = \phi \Tr(u)$ for $\phi \in C^1(\overline{\mathcal{D}})$ and $u\in W^{1,1}(\mathcal{D})$. Indeed, for $u\in C^1(\overline{\mathcal{D}})$ this is clear, and the general case follows by approximation and boundedness of $\Tr:W^{1,1}(\mathcal{D})\to L^1(\partial D)$.
    Thus, by integration by parts
    \begin{align*}
    \int_{\mathcal{D}} b^{ik} (\partial_i v)  v dx = \frac12 \int_{\mathcal{D}} b^{ik} (\partial_i v^2)  dx = \frac12 \int_{\partial\mathcal{D}}  b^{ik}  \Tr(v^2) n_i  dS + \frac12 \int_{\mathcal{D}} (\partial_i b^{ik})  v^2  dx,
    \end{align*}
    where $n$ denotes the outer normal of $\mathcal{D}$.
    Taking sums over $i$ and $\ell^2$-norms in $k$ for the last term we can write
    \begin{align*}
    \Big\|k\mapsto \sum_{i=1}^d \int_{\mathcal{D}} (\partial_i b^{ik})  v^2  dx\Big\|_{\ell^2} \leq \int_{\mathcal{D}} \|{\rm div}(b)\|_{\ell^2}  v^2  dx\leq D_b^2 \|v\|_{L^2(\mathcal{D})}^2.
    \end{align*}
    For the boundary term we obtain
    \begin{align*}
    \Big\|k\mapsto \sum_{i=1}^d \int_{\partial\mathcal{D}}  b^{ik}  \Tr(v^2) n_i  dS\Big\|_{\ell^2} \leq  \int_{\partial\mathcal{D}}  \|b\cdot n\|_{\ell^2}  \Tr(v^2)  dS
    \leq C_b \|\Tr(v^2)\|_{L^1(\partial \mathcal{D})}
    \end{align*}
    By \cite[Theorem 2.7]{Motron02} for every $\varepsilon\in (0,1)$ there exists a constant $C_{\varepsilon}>0$ such that
    \begin{align*}
    \|\Tr(v^2)\|_{L^1(\partial D)} &\leq (1+\varepsilon) \|\nabla (v^2)\|_{L^1(\mathcal{D})} + C_{\varepsilon} \|v^2\|_{L^1(\mathcal{D})}
    \\ & \leq 2(1+\varepsilon) \|v \nabla v\|_{L^1(\mathcal{D})} + C_{\varepsilon} \|v\|_{L^2(\mathcal{D})}^2
    \\ & \leq 2(1+\varepsilon) \|\nabla v\|_{L^2(\mathcal{D})}\|v\|_{L^2(\mathcal{D})} + C_{\varepsilon} \|v\|_{L^2(\mathcal{D})}^2,
    \end{align*}
    Therefore, for $v\neq 0$
    \begin{align*}
    \frac{\Big\|k\mapsto\sum_{i=1}^d \int_{\mathcal{D}} b^{ik} (\partial_i v)  v dx\Big\|_{\ell^2}}{\|v\|_{L^2(\mathcal{D})}} &\leq (1+\varepsilon) C_{b} \|\nabla v\|_{L^2(\mathcal{D})} + (C_{\varepsilon} C_{b}+D_b) \|v\|_{L^2(\mathcal{D})}.
    \end{align*}
    Taking squares we obtain the desired estimate by using $(x+y)^2\leq (1+\varepsilon) x^2+ C_{\varepsilon}'y^2$, and by redefining $\varepsilon$.
    \end{proof}

\begin{proof}[Proof of Proposition~\ref{prop:Neumann}]
    We show that Assumptions ~\ref{ass: stoch_heat_eq_neumann} \condref{it:hem}-\condref{it:bound2} hold, where we set $\phi = 0, \psi = 0$. Application of Theorem~\ref{Main_theorem} gives the result. We see that \condref{it:hem}, \condref{it:bound1} and \condref{it:bound2} are similar to the proof of Proposition~\ref{prop: stoch_heat_equation_dir}. To prove \condref{it:weak_mon}, we require an extra step in inequality \eqref{eq: stoch_heat_equation_dir_H2}. Note that the same sequence of inequalities hold, since we only use the uniform ellipticity condition. This condition also follows from the new uniform ellipticity condition in Assumptions ~\ref{ass: stoch_heat_eq_neumann}. By linearity of the operators, it suffices to consider $v \in H^1(\mathcal{D})$. Using inequality \eqref{eq: stoch_heat_equation_dir_H2}, this results in:
    \begin{align*}
   2\langle A(t, v), v\rangle + \sum\limits_{k=1}^\infty \Big\|\sum\limits_{i=1}^d b_k^{i}\partial_i v\Big\|_{L^2(\mathcal{D})}^2
            &\leq -\theta\sum\limits_{i=1}^d \int_\mathcal{D} |\partial_i v|^2 \dint x \\ & \leq -\theta\|v\|_{H^1(\mathcal{D})}^2 + \theta \|v\|_{L^2(\mathcal{D})}^2.
     \end{align*}
    We are left to prove \condref{it:coerc}. It only remains to inspect the term $\|B_t(v)^*v\|/\|v\|^2$ where $v\in H^1(\mathcal{D})$. Let $\varepsilon \in (0, 1)$ and invoke lemma \eqref{lemma: neumann_coercivity} to produce the  bound
    \begin{equation*}
        \begin{split}
            &2\langle A(t, v), v\rangle + \sum\limits_{k=1}^\infty \Big\|\sum\limits_{i=1}^d b_k^{i}\partial_i v\Big\|_{L^2(\mathcal{D})}^2 + (p-2)\frac{\|B_t(v)^*v\|_{\ell^2}^2}{\|v\|_{L^2(\D)}^2}\\
            &\leq (-\theta + (p-2)\varepsilon C_b^2)\|\nabla v\|_{L^2(\D)}^2 + C_\varepsilon(C_b^2+D_b^2) \|v\|_{L^2(\D)}^2\\
            &\leq (-\theta + (p-2)\varepsilon C_b^2)\|v\|_{H^1(\D)}^2 + (-(1+\varepsilon)C_b^2 + \theta + C_\varepsilon(C_b^2+D_b^2)) \|v\|_{L^2(\D)}^2,
        \end{split}
    \end{equation*}
    where $v \neq 0$. Since $C_b\in L^\infty(\Omega)$, we can choose $\varepsilon>0$ such that $(p-2)\varepsilon C_b^2)\leq \theta/2$, and this gives \condref{it:coerc}. Applying Theorem~\ref{Main_theorem} the required statement follows.
\end{proof}

\subsection{Stochastic Burgers' equation with Dirichlet boundary conditions}\label{stoch_burg}
We consider Burgers' equation with multiplicative gradient noise, that is,
\begin{equation}\label{eq:Burgers_equation}
    \dint u(t) = \left(\partial^2 u(t) + u(t) \partial u(t) \right) \dint t + \gamma \partial u(t) \dint W(t), \quad x \in (0,1),
\end{equation}
where $W(t)$ is a real-valued Wiener process. Equation~\eqref{eq:Burgers_equation} was first studied in \cite{brzezniak_1991} and subsequently in \cite{DaPrato_1994} with space-times white noise. We consider the same setting treated in \cite[Example 6.3]{neelima_2020} of a one-dimensional gradient noise term. The novelty is that our main abstract theorem allows to treat arbitrary moments in $\Omega$ using the classical parabolicity condition.
\begin{assumptions}\label{ass:burgers_assumptions}
    Let $\gamma \in (-\sqrt 2, \sqrt 2)$, $T > 0$, and
    \[
        (V, H, V^*) = (H_0^1(0, 1), L^2(0, 1), H^{-1}(0, 1))
    \]
    and take $U = \R$.
\end{assumptions}
Now \eqref{eq:Burgers_equation} can be reformulated as a stochastic evolution equation
\begin{equation*}
    \dint u(t) = A(u(t)) \dint t + B(u(t)) \dint W(t),
\end{equation*}
where $A:H_0^1(\Distr) \to H^{-1}(\Distr)$ is given by
\begin{equation}\label{A_burgers}
    \langle A(u), v\rangle = -\int_0^1 \partial u \, \partial v \, \dint x + \int_0^1 u \, \partial u \, v \, \dint x \qquad \text{for } u, v \in H_0^1(0, 1),
\end{equation}
and $B \colon H_0^1(0, 1) \to L^2(0, 1)$ is defined by
\begin{equation}\label{B_burgers}
    B(v) = \gamma \partial v \qquad \text{for } v \in H_0^1(0, 1).
\end{equation}
Note that in order to align with our abstract framework, we would have to take $B \colon H_0^1(0, 1) \to \mathcal L_2(\R,L^2(0, 1))$ but we do not distinguish between the through the multiplication operation trivially isomorphic spaces $\mathcal L_2(\R,L^2(0, 1))$ and $L^2(0, 1)$.

Since we can allow $p=\infty$ in the above, it will turn out that we are able to obtain uniform estimates in $\Omega$ for this particular example. This is in correspondence with what has been shown in \cite{brzezniak_1991}, albeit obtained in a different way.

\begin{proposition}
    Suppose that Assumptions~\ref{ass:burgers_assumptions} are satisfied. Let $p\in [4, \infty]$. Then, for any $u_0 \in L^{p}(\Omega; L^2(0, 1))$ the equation \eqref{eq:Burgers_equation} has a unique solution $u$, and  the following energy estimate holds
    \begin{equation*}
        \|u(t)\|_{L^p(\Omega;C([0,T];L^2(0,1)))} + \|u\|_{L^p(\Omega;L^2(0,T;H^1_0(0,1)))} \leq C\|u_0\|_{L^p(\Omega; L^2(0,1))},
    \end{equation*}
        where $C$ only depends on $\gamma$.
\end{proposition}

\begin{proof}
    As in previous instances, it suffices to verify Assumptions ~\ref{main_assumptions}, \condref{it:hem}-\condref{it:bound2}, with $f = 0$ and $K_B = K_c = 0$, so that the proposition follows from Theorem~\ref{Main_theorem} and Corollary~\ref{cor:a_priori_remark}. Hemicontinuity \condref{it:hem} is obvious. In order to prove local weak monotonicity \condref{it:weak_mon}, note that for $u, v \in H_0^1(0, 1)$ we have
    \begin{align*}
        \langle A(u)-A(v), u-v\rangle & \stackrel{\eqref{A_burgers}}{=} -\int_0^1 \partial(u-v) \, \partial(u-v) \, \dint x + \int_0^1 (u \, \partial u - v \, \partial v) \, (u-v) \, \dint x \\
                                      & \, = - \|u-v\|_{H_0^1(0, 1)}^2 - \frac 1 2 \int_0^1 (u-v) \, \partial (u^2-v^2) \, \dint x.
    \end{align*}
    Integration by parts then entails
    \begin{align*}
        - \frac 1 2 \int_0^1 (u-v) \, \partial (u^2-v^2) \, \dint x & = \frac{1}{2}\int_0^1 (u^2-v^2) \, \partial(u-v) \, \dint x                                         \\                                                                     & = \frac{1}{6} \int_0^1 \partial(u-v)^3 \, \dint x + \int_0^1 v \, (u-v) \, \partial(u-v) \, \dint x \\
        & = \int_0^1 v (u-v) \, \partial(u-v) \, \dint x,
    \end{align*}
    so that
    \begin{align*}
        \langle A(u)-A(v), u-v\rangle & = -\|u-v\|_{H_0^1(0, 1)}^2 - \int_0^1 v (u-v) \, \partial(u-v) \, \dint x                    \\
                                      & \leq -\|u-v\|_{H_0^1(0, 1)}^2 + \|v\|_{L^4 (0, 1)} \|u-v\|_{L^4(0,1)} \|u-v\|_{H_0^1(0, 1)}.
    \end{align*}
    We employ the Sobolev-Gagliardo-Nirenberg and Poincar\'{e} inequality to obtain
    \begin{equation*}
        \|v\|_{L^4(0, 1)} \leq C \|v\|_{L^2(0, 1)}^{\frac{3}{4}} \|v\|_{H_0^1(0, 1)}^{\frac{1}{4}}\leq C'\|v\|_{L^2(0, 1)}^{\frac{1}{2}} \|v\|_{H_0^1(0, 1)}^{\frac{1}{2}},
    \end{equation*}
    so that
    \begin{align*}
        \langle A(u)-A(v), u-v\rangle
         & \leq -\|u-v\|_{H_0^1(0, 1)}^2 + \|v\|_{L^4(0, 1)} \|u-v\|_{L^2(0, 1)}^{\frac{1}{2}}\|u-v\|_{H_0^1(0, 1)}^{\frac{3}{2}}                                  \\
         & \stackrel{\mathrm{(i)}}{\leq} (\varepsilon-1)\|u-v\|_{H_0^1(0, 1)}^2 + C_{\varepsilon} \|v\|_{L^4(0, 1)}^4 \|u-v\|_{L^2(0, 1)}^2                        \\
         & \stackrel{\mathrm{(ii)}}{\leq} (\varepsilon-1)\|u-v\|_{H_0^1(0, 1)}^2 + C_{\varepsilon}\|v\|_{L^2(0, 1)}^2 \|v\|_{H_0^1(0, 1)}^2 \|u-v\|^2_{L^2(0, 1)},
    \end{align*}
    where (i) follows from Young's inequality for some $\varepsilon \in (0, 1)$ and (ii) is a consequence of the Sobolev-Gagliardo-Nirenberg inequality. Now we combine with \eqref{B_burgers} to get
    \begin{equation*}
        \begin{split}
            &2\langle A(u)-A(v), u-v\rangle + \|B(u)-B(v)\|_{L^2(0, 1)}^2 \\ & \leq (\gamma^2 + 2\varepsilon-2)\|u-v\|_{H_0^1(0, 1)}^2
             + C_\varepsilon\big(1+\|v\|_{L^2(0, 1)}^2\big)\big(1+\|v\|_{H_0^1(0, 1)}^2\big)\|u-v\|_{L^2(0, 1)}^2,
        \end{split}
    \end{equation*}
    where $u, v\in H_0^1(0, 1)$. Noting that $\gamma \in (-\sqrt 2, \sqrt 2)$ and taking $\varepsilon = \frac{2-\gamma^2}{2}$, \condref{it:weak_mon} holds with $K = C_\varepsilon$, $\alpha = 2$, and $\beta = 2$.

    For coercivity \condref{it:coerc}, we first inspect the quantity $\frac{\|B(v)^*v\|_U^2}{\|v\|_H^2}$ with $v\in H_0^1(0, 1)$ and $v\neq 0$. Now, note that the following holds by using integration by parts
    \begin{equation*}
        \|B(v)^*v\|_U^2 = \gamma \int_0^1 v \, \partial v \, \dint x = \frac{\gamma}{2} \int_0^1 \partial (v^2) \, \dint x = 0.
    \end{equation*}
    By \eqref{A_burgers} and \eqref{B_burgers},  this leads to
    \begin{equation*}
        \begin{split}
            2\langle A(v), v\rangle + \|B(v)\|_{L^2(0,1)}^2 + & (p-2)\frac{\|B(v)^*v\|_U^2}{\|v\|_H^2} \\ & =  -2\|v\|_{H_0^1(0, 1)}^2 + \int_0^1 v^2 \, \partial v \, \dint x + \gamma^2\|v\|_{H_0^1(0, 1)}^2.\\
        \end{split}
    \end{equation*}
    Since $\int_0^1 v^2 \, \partial v \, \dint x = \frac 1 3 \int_0^1 \partial (v^3) \, \dint x = 0$,
    we get
    \begin{equation*}
        2\langle A(v), v\rangle + \|B(v)\|_{L^2(0, 1)}^2 + (p-2)\frac{\|B(v)^*v\|_U^2}{\|v\|_H^2} = (-2+\gamma^2)\|v\|_{H_0^1(0, 1)}^2.
    \end{equation*}
    Therefore, \condref{it:coerc} holds with $\theta = 2 - \gamma^2 > 0$, $\alpha = 2$, $f = 0$, and $K_c = 0$

    Let $u, v\in H_0^1(0, 1)$. For the boundedness condition \condref{it:bound1}, observe
    \begin{equation*}
        |\langle A(u), v \rangle| \leq \int_0^1 |\partial u| \, |\partial v| \, \dint x + \Big|\int_0^1 u \, \partial u \, v \, \dint x\Big|,
    \end{equation*}
    where
    \begin{equation*}
        \int_0^1 |\partial u| \, |\partial v| \, \dint x \leq \|u\|_{H_0^1(0, 1)} \|v\|_{H_0^1(0, 1)}
    \end{equation*}
    by the Cauchy-Schwarz inequality and
    \begin{equation*}
        \begin{split}
            \Big|\int_0^1 u \, \partial u \, v \, \dint x \Big| & \ \stackrel{\mathrm{(i)}}{=} \Big|\int_0^1 \frac{1}{2} (u^2) \, \partial v \, \dint x\Big| \stackrel{\mathrm{(ii)}}{\le} \frac{1}{2} \|u\|_{L^4(0, 1)}^2 \|v\|_{H_0^1(0, 1)} \\
            &\stackrel{\mathrm{(iii)}}{\le} C \|u\|_{L^2(0, 1)}\|u\|_{H_0^1(0, 1)}\|v\|_{H_0^1(0, 1)},
        \end{split}
    \end{equation*}
    where we have applied integration by parts in (i), H\"{o}lder's inequality in (ii), and the Sobolev-Gagliardo-Nirenberg inequality in (iii). This results in
    \begin{equation*}
        \left|\langle A(u), v\rangle\right| \leq \big(\|u\|_{H_0^1(0, 1)} + C\|u\|_{L^2(0, 1)}\|u\|_{H_0^1(0, 1)}\big)\|v\|_{H_0^1(0, 1)},
    \end{equation*}
    Using $\alpha = 2$ as in \condref{it:weak_mon} and \condref{it:coerc}, we obtain
    \begin{equation*}
        \|A(u)\|_{H^{-1}(0, 1)}^2 \leq C' \|u\|_{H_0^1(0, 1)}^2 \big(1 + \|u\|_{L^2(0, 1)}^2\big),\\
    \end{equation*}
    proving \condref{it:bound1} with $K_A = C'$ and $\beta = 2$. Finally, for $v \in H_0^1(0, 1)$, $\|B(v)\|_{L^2(0, 1)}^2 = \gamma^2 \|v\|_{H_0^1(0, 1)}^2$, so that \condref{it:bound2} is satisfied with $K_B = 0$ and $K_\alpha = \gamma^2$.
\end{proof}

\subsection{Stochastic Navier-Stokes equations in 2D}\label{stoch_nav_sto}
Consider the stochastic Navier-Stokes equations in two space dimensions with multiplicative gradient noise
\begin{equation}\label{stoch_navier}
    \dint u(t) = (\nu \Delta u(t) - (u, \nabla )u)\dint t + \sum\limits_{k=1}^\infty [(b_k, \nabla)u] \dint W_k(t) - (\nabla p) \dint t.
\end{equation}
Here, $(W_k(t))_{t\geq 0}$ is a collection of independent real Wiener processes indexed by $k\in\mathbb{N}$. The components $b_k$ are set to be vectors of divergence free vector fields (see Assumptions~\ref{navier_stokes_assumptions} below). Equation~\eqref{stoch_navier} was considered in \cite{brzezniak_1991} using semigroup methods, and later on in many other papers (see \cite{AV20_NS} and references therein). For simplicity we do not consider additional forcing terms, but they can be included without difficulty (see Remark \ref{rem:additive}).

In what follows, we use:
\begin{assumptions}\label{navier_stokes_assumptions}
    Suppose $\mathcal{D} \subseteq \mathbb{R}^2$ is a bounded domain. Furthermore, assume $\nu > 0$, $T > 0$,
    $(b_k)_{k \in \N} \in L^\infty((0,T)\times \Omega\times \mathcal D;\ell^2(\N;\R^{2\times2}))$ which is progressively measurable, and satisfies $\DIV b_k = (\sum_{i = 1}^2 \partial_i b_k^{i\gamma})_{\gamma = 1}^2 = 0$ in the sense of distributions for all $k \in \N$. We impose the coercivity condition that there exists $\kappa > 0$ such that
    \begin{equation}\label{coercivity_assumption_navier}
        \Big(2\nu \sum\limits_{i, \gamma = 1}^2 (\xi^{i, \gamma})^2  - \sum\limits_{k=1}^\infty\sum\limits_{\gamma, \gamma' = 1}^2 \sum\limits_{i, j = 1}^2 b_k^{i\gamma} b_k^{j\gamma'} \xi^{i, \gamma} \xi^{j, \gamma'} \Big) \geq \kappa \sum\limits_{i, \gamma = 1}^2 (\xi^{i, \gamma})^2,
    \end{equation}
    for all $\xi \in \mathbb{R}^{2\times 2}$.
    Set $U := \ell^2$ and define $(V,H,V^*)$ by
    \begin{equation*}
        V = \{v \in W_0^{1, 2}(\mathcal{D}; \mathbb{R}^2) : \nabla \cdot v = 0 \quad a.e. \text{ on } \mathcal{D}\}, \quad \|v\|_V := \Big(\int_{\mathcal{D}}|\nabla v|^2 \dint x\Big)^{\frac{1}{2}},
    \end{equation*}
    and where $H$ is the closure of $V$ with respect to the norm
    \begin{equation*}
        \|v\|_H := \Big(\int_{\mathcal{D}}|v|^2 \dint x\Big)^{\frac{1}{2}}.
    \end{equation*}
\end{assumptions}
Defining the Helmholtz-Leray projection $\mathbb{P}_\mathrm{HL}$ as the orthogonal projection
\begin{equation*}
    \mathbb{P}_\mathrm{HL}: L^2(\mathcal{D}; \mathbb{R}^2) \to H,
\end{equation*}
equation~\eqref{stoch_navier} turns into a stochastic evolution equation
\begin{equation}\label{reduced_navier}
    \dint u(t) = (Lu(t)+ F(u(t)))\dint t + \sum\limits_{k=1}^\infty B_k(u(t)) \dint W_k(t),
\end{equation}
where $L: H^{2, 2}(\Distr; \R^2) \cap V \to H$ is given by
$$Lu = \nu \mathbb{P}_\mathrm{HL} (\Delta u), \quad u \in H^{2, 2}(\Distr; \R^2) \cap V$$
and can be extended to a map $L: V\to V^*$ such that $\|Lu\|_{V^*} \leq \|u\|_V, u \in V$. Furthermore, set $F$ to be a nonlinear operator $F: V \to V^*$ given by
$$F(u) = -\mathbb{P}_\mathrm{HL}[(u, \nabla)u] = - \mathbb P_\mathrm{HL}[\DIV (u \otimes u)], \quad u \in V.$$
Finally, define $B \colon V \to \mathcal L_2(U,H)$ by
\[
    B(u)e_k = B_k(u) = \mathbb{P}_\mathrm{HL}[(b_k, \nabla) u], \quad u \in V.
\]
\begin{theorem}\label{thm:SNS}
    Suppose Assumption~\ref{navier_stokes_assumptions} holds and let $p \in[2, \infty]$. Then, for any $u_0 \in L^{p}(\Omega; H)$, there exists a unique solution $u$ to equation \eqref{reduced_navier} and there exists a constant $C$ only depending on $\kappa$ such that
    \begin{equation*}
        \|u\|_{L^p(\Omega;C([0,T];H))} + \|u\|_{L^p(\Omega;L^2(0,T;V))} \leq C\|u_0\|_{L^p(\Omega; H)}
    \end{equation*}
\end{theorem}

\begin{remark}
    The special case of periodic boundary conditions in case of rough initial data was recently considered in \cite{agresti2021stochastic}, where high order regularity was proved. There the monotone operator setting (in $L^2(\Omega)$) was combined with a new approach to SPDEs based on maximal regularity techniques (see \cite{AV19_QSEE_1, AV19_QSEE_2}). The main difficulty to prove high order regularity for the solution to \eqref{stoch_navier} is that the nonlinearity is {\em critical} for the space $L^2(0,T;V)$. Therefore, classical bootstrapping arguments do not give any regularity.
\end{remark}

\begin{proof}[Proof of Theorem~\ref{thm:SNS}]
    It suffices to use Theorem \ref{Main_theorem} and Corollary~\ref{cor:a_priori_remark} with $A(u) := L u + F(u)$, for which \condref{it:hem}-\condref{it:bound2} under the assumption $K_B = K_c = 0$ have to be shown. We will only show \condref{it:weak_mon}, \condref{it:coerc} and \condref{it:bound2}. For the other assumptions we refer to \cite{brzezniak_strong_2014,Rockner_SPDE_2015}. In order to show local monotonicity \condref{it:weak_mon}, let $u, v \in V$. The quantity $\langle Lu-Lv, u-v\rangle$ can be computed from the definition
    \begin{equation}\label{eq:lin_navier}
        \langle Lu-Lv, u-v\rangle = - \nu \|u-v\|_V^2.
    \end{equation}
    Next, we compute
    \[
        \langle F(u)-F(v), u-v\rangle = - \langle \DIV((u-v) \otimes v), u-v\rangle - \langle \DIV(u \otimes (u-v)), u-v\rangle,
    \]
    where
    \begin{equation}\label{eq:nse_non_0}
        \langle \DIV(u \otimes (u-v)), u-v\rangle = - \frac 1 2 \langle \nabla |u-v|^2, u \rangle = 0
    \end{equation}
    and the Sobolev-Gagliardo-Nirenberg inequality entails
    \[
        - \langle \DIV((u-v) \otimes v), u-v\rangle \le C \|u-v\|_V^{\frac 3 2} \|u-v\|_H^{\frac 1 2} \|v\|_{L^4(\mathcal{D};\mathbb{R}^2)}
    \]
    for a constant $C$, so that by Young's inequality
    \begin{equation}\label{eq:nonlin_navier}
        \langle F(u)-F(v), u-v\rangle \leq \kappa \|u-v\|_V^2+\frac{C'}{\kappa^3}\|v\|_{L^4(\mathcal{D};\mathbb{R}^2)}^4\|u-v\|_H^2
    \end{equation}
    for a constant $C'$. Finally, the contribution \condref{it:weak_mon} coming from the stochastic integral is
    \begin{equation}\label{eq:stoch_navier}
        \begin{split}
            \sum\limits_{k=1}^\infty \|B_k(u)-B_k(v)&\|^2_{H} \stackrel{(\mathrm{i})}{\leq} \sum\limits_{k=1}^\infty \|[(b_k, \nabla)(u-v)]\|^2_H\\
            &\stackrel{(\mathrm{ii})}{=} \sum\limits_{k=1}^\infty \sum\limits_{\gamma,\gamma' =1}^2\sum\limits_{i, j = 1}^2 \int_{\mathcal{D}} b_k^{i\gamma}b_k^{j\gamma'}\partial_i (u-v)^{\gamma} \partial_j (u-v)^{\gamma'} \dint x,
        \end{split}
    \end{equation}
    where (i) follows since projections are contractive and (ii) is the first line written out. By \eqref{eq:lin_navier}, \eqref{eq:nonlin_navier}, \eqref{eq:stoch_navier} and the coercivity condition \eqref{coercivity_assumption_navier} we obtain
    \begin{align*}
         & 2\langle L u + F(u) - (L v +F(v)), u-v\rangle + \sum\limits_{k=1}^\infty \|B_k(u)-B_k(v)\|_H^2 \\
         & \leq \frac{C'}{\kappa^3}\|v\|_{L^4(\mathcal{D}; \mathbb{R}^2)}^4 \|u-v\|_H^2                   \\
         & \stackrel{\mathrm{(i)}}{\leq} \frac{C''}{\kappa^3} \|v\|_V^2 \|v\|_H^2 \|u-v\|_H^2             \\
         & \leq  \frac{C''}{\kappa^3}(1+\|v\|_V^2)(1+\|v\|_H^2)\|u-v\|_H^2,
    \end{align*}
    with a constant $C''$ and (i) follows from the Sobolev-Gagliardo-Nirenberg inequality. The above implies that \condref{it:weak_mon} holds with $\alpha = \beta = 2$ and $K = \frac{C''}{\kappa^3}$.

    In order to show \condref{it:coerc}, note that
    \begin{equation*}
        \langle Lv, v \rangle \stackrel{\eqref{eq:lin_navier}}{=} -\nu \|v\|_V^2, \qquad v \in V.
    \end{equation*}
    We also note that $\langle F(v), v\rangle \stackrel{\eqref{eq:nse_non_0}}{=} 0$ for $v \in V$. Therefore, the only term that remains to be estimated is
        $\frac{\|B^*(u)u\|_{U}^2}{\|u\|_H^2}$.
    This will also turn out to be 0, by using that the components of $b^k$ are divergence free vector fields. Indeed, we obtain for $k \in \N$,
    \begin{equation*}
        \begin{split}
            &(B^*(v)v)_k= \int_{\mathcal{D}}\left[(b_k, \nabla) v\right]\cdot v \ \dint x\\
            &= \underbrace{\int_{\mathcal{D}} \big((b_k^{11}\partial_1 v^1) v^1 + (b_k^{12} \partial_2 v^1) v^1\big) \dint x}_{\boxed{\text{A}}} + \int_{\mathcal{D}} \big((b_k^{21} \partial_1 v^2)v^2 + (b_k^{22} \partial_2 v^2)v^2\big) \dint x.
        \end{split}
    \end{equation*}
    By renumbering, it suffices to treat \boxed{A}. Using integration by parts, we see:
    \begin{equation*}
        \begin{split}
            \boxed{\text{A}} = \frac 1 2 \int_{\mathcal{D}} \big(b_k^{11} \partial_1 (v^1)^2 + b_k^{12} \partial_ 2 (v^1)^2\big) \dint x = \frac 1 2 \int_{\mathcal D} (\partial_1 b_k^{11} + \partial_2 b_k^{12}) (v^1)^2 \dint x = 0 \\
        \end{split}
    \end{equation*}
    and thus $(B^*(v)v)_k = 0$ for all $k\in \mathbb{N}$. We therefore conclude that the coercivity condition \condref{it:coerc} is as follows:
    \begin{equation*}
        2\langle Lv + F(v), v\rangle + \sum\limits_{k=1}^\infty \|B_k(v)\|_{H}^2 \leq -\nu \|v\|_V^2, \qquad v \in V,
    \end{equation*}
    that is, we can choose $\theta = \kappa$, $f(t) = 0$, and $K_c = 0$.

    In order to show \condref{it:bound2}, we use \eqref{coercivity_assumption_navier} and \eqref{eq:stoch_navier} once more and arrive at
    \begin{equation*}
        \sum\limits_{k=1}^\infty \|B_k(v)\|_{H}^2 \leq (2\nu -\kappa)\|v\|_V^2,
    \end{equation*}
    showing that also \condref{it:bound2} holds on choosing $K_B = 0$ and $K_\alpha = 2\nu-\kappa$.
\end{proof}

\subsection{Systems of second order SPDEs}\label{systems_SPDE}
The authors of \cite{du_2020} develop
a $C^{2+\delta}$ theory for systems of SPDEs.
This relies on integral estimates for a model system of SPDEs (see \cite[Theorem~3.1]{du_2020}). We will show that one of the underlying assumptions, which the authors of \cite{du_2020} call the \textit{modified stochastic parabolicity condition}, fits naturally in our framework.
Sharpness follows from \cite[Example 1.1]{du_2020} which is based on \cite[Section 3]{KimLee}.

Consider a random field $$\mathbf{u} = (u^1, ..., u^N)': \mathbb{R}^d \times [0, \infty) \times \Omega \to \mathbb{R}^N$$ described by the following linear system of SPDEs:
\begin{equation}\label{eq: KaiDu_SPDE}
    \dint u^\alpha = \left(a^{ij}_{\alpha\beta} \partial_{ij} u^\beta + \phi_\alpha\right)\dint t + \left(\sigma^{i}_{k, \alpha\beta}\partial_i u^\beta + \psi_{k, \alpha}\right) \dint W_k(t)
\end{equation}
where the collection $\{W_k\}_{k\geq 1}$ are countably many independent Wiener processes.

In this section we use Einstein's summation convention with $$i, j = 1, 2, ..., d; \quad \alpha, \beta = 1, 2, ..., N; \quad k = 1, 2, ...$$ The assumptions are:
\begin{assumptions}\label{system_assumptions}
    Let $p\in [2,\infty)$, $d\geq 1$ and $N\geq 1$. Let
    \begin{equation*}
        (V, H, V^*) = (H^{m+1}(\R^d; \R^N), H^m(\R^d; \R^N), H^{m-1}(\R^d; \R^N))
    \end{equation*}
    and $U = \ell^2$.
    Further assume that $a^{ij}_{\alpha\beta} \in  L^\infty(\Omega \times [0, T])$ for all $1 \leq i, j \leq d$, $1 \leq \alpha, \beta \leq N$ and $(\sigma^{i}_{k, \alpha\beta})_{k=1}^\infty \in L^\infty(\Omega \times [0, T]; \ell^2)$ for all $1 \leq i \leq d$, $1 \leq \alpha, \beta \leq N$.
    and suppose that the following stochastic modified parabolicity condition is satisfied:
    \begin{itemize}
        \item[(MSP)] The coefficients $a = (a^{ij}_{\alpha\beta})$ and $\sigma = (\sigma^{i}_{k, \alpha\beta})$ are said to satisfy the stochastic modified parabolicity (MSP) condition if there are measurable functions $\lambda^{i}_{k, \alpha\beta}: \mathbb{R}^d \times [0, \infty)\times \Omega \to \mathbb{R}$ with $\lambda^{i}_{k, \alpha\beta} = \lambda^{i}_{k, \beta\alpha}$ such that for
              \begin{equation*}
                  \mathcal{A}^{ij}_{\alpha\beta} = 2a^{ij}_{\alpha\beta} -\sigma^{i}_{k, \gamma\alpha}\sigma^{j}_{k, \gamma\beta} - (p-2)(\sigma^{i}_{k, \gamma\alpha}-\lambda^{i}_{k, \gamma\alpha})(\sigma^{j}_{k, \gamma\beta}-\lambda^{j}_{k, \gamma\beta})
              \end{equation*}
              there exists a constant $\kappa > 0$ with
              \begin{equation*}
                  \mathcal{A}^{ij}_{\alpha\beta} \xi_i \xi_j \eta^\alpha\eta^\beta \geq \kappa |\xi|^2 |\eta|^2 \quad \forall \xi\in\mathbb{R}^d, \eta\in\mathbb{R}^N
              \end{equation*}
              everywhere on $\mathbb{R}^d \times [0, \infty) \times \Omega$.
    \end{itemize}
    Suppose that $u_0\in L^{p}(\Omega,\F_0;H)$ and
    \[
        \phi \in L^p(\Omega; L^2([0, T]; \\ H^{m-1}(\mathbb{R}^d; \mathbb{R}^N))), \quad
        \psi\in L^p(\Omega; L^2([0, T]; H^m(\mathbb{R}^d; \ell^2(\mathbb{N}; \mathbb{R}^N)))).
    \]
\end{assumptions}
\begin{remark}
The above ellipticity condition $\mathcal{A}^{ij}_{\alpha\beta} \xi_i \xi_j \eta^\alpha\eta^\beta \geq \kappa |\xi|^2 |\eta|^2$ is known as the Legendre-Hadamard condition. In case the coefficients depend on the space variable some smoothness is required if one wishes to assume this type of ellipticity. Alternatively, one can consider measurable coefficients with a more restrictive ellipticity condition. For details on these matters we refer to \cite{AHMT}.

In the MSP condition one typically takes
$\lambda^{i}_{k, \alpha\beta} = (\sigma^{j}_{k, \alpha\beta} + \sigma^{j}_{k, \beta\alpha})/2$ or $\lambda^{i}_{k, \alpha\beta} = 0$.
\end{remark}

We can reformulate \eqref{eq: KaiDu_SPDE} as a stochastic evolution equation
\begin{equation}\label{reduced_kaidu}
    \dint u(t) = A(u(t)) \dint t + \sum\limits_{k=1}^\infty B_k(u(t)) \dint W_k(t).
\end{equation}
For this, define the deterministic part of the equation as an operator
\[
    A\colon H^{m+1}(\mathbb{R}^d; \mathbb{R}^N) \to H^{m-1}(\mathbb{R}^d; \mathbb{R}^N)
\]
such that for any $u, v \in H^{m+1}(\mathbb{R}^d; \mathbb{R}^N)$
\begin{equation}\label{System_A}
    \langle A(u), v \rangle = -\int_{\mathbb{R}^d} a^{ij}_{\alpha\beta} \partial_i u^\beta \partial_j u^\alpha \dint x.
\end{equation}
The stochastic part of the equation is defined as an operator
\[
    B\colon H^{m+1}(\mathbb{R}^d; \mathbb{R}^N) \to \mathcal L_2(\ell^2; H^m(\mathbb{R}^d; \mathbb{R}^N))
\]
such that for any $u \in H^{m+1}(\mathbb{R}^d; \mathbb{R}^N)$ and
\begin{equation}\label{System_B}
    B(u)e_k =  B_k(u) \quad \text{with} \quad B_{k, \alpha} (u) = \sigma^{i}_{k, \alpha\beta}\partial_i u^\beta.
\end{equation}
We are now in a position to recover \cite[Theorem~3.1]{du_2020}:
\begin{proposition}\label{prop:system}
    Let $m\geq 0$, and suppose that Assumptions~\ref{system_assumptions} are satisfied.
    Then, \eqref{eq: KaiDu_SPDE} has a unique solution
    \[
        u\in L^p(\Omega; C([0, T]; H^m(\mathbb{R}^d; \mathbb{R}^N)))\cap L^p(\Omega; L^2([0, T]; H^{m+1}(\mathbb{R}^d; \mathbb{R}^N))).
    \]
Moreover, for any multi-index $\mathfrak{s}$ with $|\mathfrak{s}| \leq m$, there exists a constant $C$ depending on $d$, $\kappa$ and $K$ such that
    \begin{align*}
         & \E \sup\limits_{t\in[0, T]} \|\partial^{\mathfrak{s}} u(t)\|^p_{L^2(\mathbb{R}^d; \mathbb{R}^N)} + p^{-1/2} \E \Big(\int_0^T \|\partial^{\mathfrak{s}} \partial_x u(t)\|_{L^2(\mathbb{R}^d; \mathbb{R}^N)}^2 \dint t\Big)^\frac{p}{2}                                                                 \\
         & \quad \leq C \Big(\E\|\partial^{\mathfrak{s}} u_0\|_{L^2(\R^d; \R^N)}^p + \E\Big(\int_0^T \|\partial^{\mathfrak{s}}\phi(t)\|_{H^{-1}(\mathbb{R}^d; \mathbb{R}^N)}^2 \dint t \Big)^{\frac{p}{2}}\\
         & \qquad + \E \Big(\int_0^T \|\partial^{\mathfrak{s}}\psi(t) \|_{L^2(\mathbb{R}^d; \ell^2(\mathbb{N}; \mathbb{R}^N))}^2 \dint t\Big)^{\frac{p}{2}}\Big).
    \end{align*}
\end{proposition}

\begin{proof}[Proof of Proposition~\ref{prop:system}]
Without loss of generality, one can restrict to the case $m = 0$, since the other cases can be obtained by differentiation. We will check the conditions of Theorem~\ref{Main_theorem} and Corollary~\ref{cor:a_priori_remark}. We proceed by showing that Assumptions ~\ref{main_assumptions} \condref{it:hem}-\condref{it:bound2} hold. By Remark~\ref{rem:additive} we may assume  $\phi = \psi = 0$. We only verify coercivity \condref{it:coerc}, since \condref{it:hem}, \condref{it:weak_mon}, \condref{it:bound1} and \condref{it:bound2} are very similar to the stochastic heat equation, to which equation~\eqref{eq: KaiDu_SPDE} reduces on setting $N=1$, and which was treated in subsections ~\ref{stoch_heat_dir} on arbitrary domains. To this end, let $v \in H^1(\mathbb{R}^d; \mathbb{R}^N)$ and consider the following (using the summation convention):
    \begin{align*}
        2\langle A(v), v \rangle & = -2\int_{\mathbb{R}^d}a^{ij}_{\alpha\beta} \partial_i v^\beta \partial_j v^\alpha \dint x.
    \end{align*}
    Next, we use definition \eqref{System_B} to consider the term $\|B_t(v)\|^2$:
    \begin{equation*}
        \begin{split}
            \|B_t(v)\|_{L^2(\mathbb{R}^d; \ell^2(\mathbb{N}; \mathbb{R}^N))}^2
            &= \int_{\mathbb{R}^d} \sigma^{i}_{k, \gamma\alpha}\sigma^{j}_{k, \gamma\beta} \partial_i v^{\beta} \partial_j v^{\alpha} \dint x
        \end{split}
    \end{equation*}
    Considering $\|B_t(v)^*v\|_{\ell^2}^2/\|v\|_{L^2(\R^d;\R^N)}^2$,
    for $v\in H^1(\mathbb{R}^d; \mathbb{R}^N)$, $v\neq 0$, we have:
    \begin{equation}\label{adjoint_term}
        \begin{split}
            \|B_t(v)^*v\|_{\ell^2}^2 &= \sum\limits_{k=1}^\infty |(B_t(v)^*v)_k|^2=\sum\limits_{k=1}^\infty\Big(\int_{\mathbb{R}^d} \sigma^{i}_{k, \gamma\beta} (\partial_ i v ^\beta) v^\gamma \dint x\Big)^2.
        \end{split}
    \end{equation}
    Note that the following identity holds:
    \begin{equation*}
        \sigma^{i}_{k, \gamma\beta}\partial_i v^\beta v^\gamma = (\sigma^{i}_{k, \gamma\beta}-\lambda^{i}_{k, \gamma\beta})v^\gamma \partial_i v^\beta + \frac{1}{2} \lambda_{k, \gamma\beta}^{i} \partial_i(v^\gamma v^\beta).
    \end{equation*}
    Integrating both sides of the above expression over $\mathbb{R}^d$, by equation \eqref{adjoint_term} we find:
    \begin{align*}
        \|B_t(v)^*v\|_{\ell^2}^2  &= \sum\limits_{k=1}^\infty\Big(\int_{\mathbb{R}^d} (\sigma^{i}_{k,\gamma\beta} - \lambda^{i}_{k,\gamma\beta}) (\partial_ i v ^\beta) v^\gamma \dint x\Big)^2                                                                                                                                   \\ & \stackrel{\mathrm{(i)}}{\leq} \sum\limits_{k=1}^\infty \Big(\int_{\mathbb{R}^d}\Big(\sum\limits_{\gamma=1}^N (v^\gamma)^2\Big)^{\frac{1}{2}} \Big( \sum\limits_{\gamma=1}^N((\sigma^{i}_{k,\gamma\beta} - \lambda^{i}_{k,\gamma\beta})\partial_i v^\beta)^2\Big)^{\frac{1}{2}} \dint x\Big)^2                              \\ & \stackrel{\mathrm{(i)}}{\leq} \sum\limits_{k=1}^\infty \|v\|^2_{L^2(\mathbb{R}^d; \mathbb{R}^N)} \Big(\int_{\mathbb{R}^d}((\sigma^{i}_{k,\gamma\beta}-\lambda^{i}_{k,\gamma\beta})\partial_i v^\beta)^2 \dint x \Big)  \\
 & = \|v\|^2_{L^2(\mathbb{R}^d; \mathbb{R}^N)} \int_{\mathbb{R}^d} (\sigma^{i}_{k,\gamma\beta}-\lambda^{i}_{k,\gamma\beta})(\sigma^{j}_{k,\gamma\alpha}-\lambda^{j}_{k,\gamma\alpha})\partial_i v^\beta \partial_j v^\alpha \dint x ,
    \end{align*}
    where the Cauchy-Schwarz inequality is applied at (i). This leads to
    \begin{equation*}
        \frac{\|B_t(v)^*v\|_{\ell^2}^2}{\|v\|^2_{L^2(\mathbb{R}^d; \mathbb{R}^N)}} \leq \sum\limits_{k=1}^\infty\Big(\int_{\mathbb{R}^d} (\sigma^{i}_{k,\gamma\beta}-\lambda^{i}_{k,\gamma\beta})(\sigma^{j}_{k,\gamma\alpha}-\lambda^{j}_{k,\gamma\alpha})\partial_i v^\beta \partial_j v^\alpha \dint x \Big).
    \end{equation*}
    Therefore, the coercivity condition \condref{it:coerc} can be derived from (MSP) as:
    \begin{equation*}
        \begin{split}
            &2\langle A(v), v\rangle + \|(B_t(v))\|_{L^2(\mathbb{R}^d; \ell^2(\mathbb{N};\mathbb{R}^N))}^2 + (p-2)\frac{\|B_t(v)^*v\|_{\ell^2}^2}{\|v\|^2_{L^2(\mathbb{R}^d; \mathbb{R}^N)}}\\
            &\leq \int_{\mathbb{R}^d} \left(-2a^{ij}_{\alpha\beta} + \sigma^{i}_{k,\gamma\alpha}\sigma^{j}_{k,\gamma\beta} + (p-2)(\sigma^{i}_{k,\gamma\beta}-\lambda^{i}_{k,\gamma\beta})(\sigma^{j}_{k,\gamma\alpha}-\lambda^{j}_{k,\gamma\alpha})\right)\partial_i v^\beta \partial_j v^\alpha \dint x\\
            &\leq -\kappa \|v\|_{H^1(\mathbb{R}^d; \mathbb{R}^N)}^2 + \kappa\|v\|_{L^2(\R^d;\R^N)}^2,
        \end{split}
    \end{equation*}
    which shows that \condref{it:coerc} holds with $\theta = \kappa$, $f = 0$, and $K_c = \kappa$.
\end{proof}

\subsection{Higher order SPDEs}\label{higher_SPDE}
In this section we consider the following on $\R^d$:
\begin{equation}\label{higher_order_spde}
    \dint u(t) = \Big[(-1)^{m+1}\sum\limits_{|\alpha|,|\beta|=m}\partial^\beta (A^{\alpha\beta})\partial^{\alpha}u(t)) + \phi(t)\Big]\dint t + \sum\limits_{k=1}^\infty \Big[\sum\limits_{|\alpha|=m}B_{k, \alpha}\partial^\alpha u(t) + \psi_{k, t}\Big] \dint W_k(t),
\end{equation}
where $(W_k(t))_{t\geq 0}$ are countably many independent Wiener processes.

The above equation was considered in \cite{wang_2020}, and below we will show that the $p$-dependent well-posedness results can be obtained within our abstract framework. Additionally, our coefficients to be space dependent.  The assumptions are:
\begin{assumptions}\label{higher_order_assumptions}
    Let $d\geq 1$, $m \geq 1$ and let
    \begin{equation*}
        (V, H, V^*) = (H^m(\mathbb{R}^d), L^2(\mathbb{R}^d), H^{-m}(\mathbb{R}^d)).
    \end{equation*}
    and take $U = \ell^2$. Further assume that the coefficients $A^{\alpha\beta} \in L^\infty(\Omega\times[0, T]\times \D)$ for all $1\leq\alpha,\beta \leq d$. Suppose that
    \[(B_{k, \alpha})_{k=1}^\infty  \in \left\{
        \begin{array}{ll}
          L^\infty(\Omega\times [0, T] \times \R^d; \ell^2), & \hbox{if $m$ is even;} \\
          W^{1, \infty}(\R^d;\ell^2), & \hbox{if $m$ is odd.}
        \end{array}
      \right.
    \]
 Assume that the coefficients satisfy the following coercivity condition:
    \begin{equation}\label{eq:wang_coercivity}
        2\sum\limits_{|\alpha|, |\beta = m} A^{\alpha\beta}\xi_\alpha\xi_\beta - \frac{p+(-1)^m(p-2)}{2}\sum\limits_{k=1}^\infty\Big|\sum\limits_{|\alpha|=m} B_{k, \alpha} \xi_\alpha \Big|^2 \geq \lambda \sum\limits_{|\alpha|=m} |\xi_\alpha|^2,
    \end{equation}
    where $\lambda >0$. Furthermore, suppose $u_0\in L^{p}(\Omega, \F_0; H)$,
    \[
        \phi\in L^p(\Omega; L^2([0, T]; \\ H^{-m}(\mathbb{R}^d))) \quad and \quad \psi\in L^p(\Omega; L^2([0, T]; L^2(\mathbb{R}^d; \ell^2))).
    \]

    \end{assumptions}
Next, we reformulate SPDE \eqref{higher_order_spde} into a stochastic evolution equation $$\dint u(t) = A(t, u(t)) \dint t + \sum\limits_{k=1}^\infty B_k(t, u(t)) \dint W_k(t).$$
The drift part of the equation is defined as a time-dependent linear operator
\[
    A(t)\colon H^{m}(\mathbb{R}^d) \to H^{-m}(\mathbb{R}^d),
\]
where for all $u, v \in H^{m}(\mathbb{R}^d)$:
\begin{equation}\label{higher_order_A}
    \langle A(t, u), v\rangle = - \sum\limits_{|\alpha|, |\beta| = m} \langle A^{\alpha\beta}\partial^\alpha u, \partial^\beta v\rangle = -\sum\limits_{|\alpha|, |\beta| = m} \int_{\mathbb{R}^d} A^{\alpha\beta} (\partial^\alpha u) (\partial^\beta v) \dint x.
\end{equation}
Similarly, the stochastic part is defined as a time-dependent linear operator
\[
    B\colon H^m(\R^d) \to \mathcal L_2(\ell^2, L^2(\mathbb{R}^d)),
\]
where for all $u \in H^{m}(\mathbb{R}^d)$
\begin{equation*}
    B(u)e_k = B_k(u)= \sum\limits_{|\alpha|=m} B_{k, \alpha}(t) \partial^\alpha u.
\end{equation*}
\begin{proposition}\label{prop:higherorder}
    Suppose that Assumption~\ref{higher_order_assumptions} is satisfied. Then, \eqref{higher_order_spde} has a unique solution in
    \[
        u \in L^p(\Omega; C([0, T]; L^2(\mathbb{R}^d))) \cap L^p(\Omega; L^2([0, T]; H^{m}(\mathbb{R}^d))).
    \]
    Furthermore, there exists a constant $C$ only depending on $\lambda$, $d$ and $p$ such that
    \begin{align*}
            &\E \sup\limits_{t\in[0, T]}\|u(t)\|_{L^2(\mathbb{R}^d)}^p + \E \Big(\int_0^T \|u(t)\|_{H^m(\mathbb{R}^d)}^2 \dint t\Big)^{\frac{p}{2}} \\ &\leq Ce^{CT} \bigg(\E\|u_0\|_{L^2(\D)}^p + \E \Big(\int_0^T \|\phi(t)\|^2_{H^{-m}(\mathbb{R}^d)}\dint t\Big)^{\frac{p}{2}} + \E \Big(\int_0^T \|\psi(t)\|_{L^2(\mathbb{R}^d;\ell^2)}^2 \dint t\Big)^{\frac{p}{2}}\bigg).
    \end{align*}
\end{proposition}
\begin{remark}
    If the coefficients are not space-dependent, one can shift the regularity as in subsection \ref{systems_SPDE}. Moreover, in that case the estimate can be obtained with more explicit constants independent of $p$ and $T$ as in Corollary \ref{cor:a_priori_remark}.
\end{remark}
\begin{remark}
    For $m$ even, no smoothness assumptions on $B$ have been made. In case $m$ is odd one can also deal with the non-smooth case, but this will require a $p$-dependent coercivity condition as in the even case. 
\end{remark}
Before starting the proof, we state a lemma needed for the coercivity condition.

\begin{lemma}\label{lemma: higher_order_coercivity}
    Suppose that $m = 2n+1$ with $n\in \N_0$, and that Assumption~\ref{higher_order_assumptions} is satisfied.
    Let $\zeta \in W^{1, \infty}(\R^d;\ell^2)$. Let $\alpha\in \N^d$ be such that $|\alpha|\leq m$. Then for every $\varepsilon>0$ there exists a $C_{\varepsilon}>0$ depending on $m$ such that for all $v\in H^m(\R^d)$
    \[
    \frac{\Big\|\int_{\R^d} \zeta  v \partial^{\alpha} v dx\Big\|_{\ell^2}}{\|v\|_{L^2(\R^d)}}\leq \varepsilon \|v\|_{H^{m}(\R^d)} + C_{\varepsilon} \|v\|_{L^2(\R^d)}.
    \]
\end{lemma}

\begin{proof}
    By density it suffices to consider $v\in C^\infty_c(\R^d)$.
    If $|\alpha|=m$, then we reduce the number of derivatives by one order.
    Integrating by parts $|\alpha|$ times we obtain
    \begin{align*}
    \int_{\R^d} \zeta_k  v \partial^{\alpha} v dx = -\int_{\R^d} \zeta_k  v\partial^{\alpha}v dx + R.
    \end{align*}
    where $R_k$ is a linear combination of terms of the form
    $\int_{\R^d} \partial^{\wt{\alpha}} \zeta_k  \partial^{\beta} v \partial^{\gamma} v dx$ with $|\wt{\alpha}| + |\beta| + |\gamma| = |\alpha|$ and $|\wt{\alpha}| = 1$. Therefore,
    $\int_{\R^d} \zeta_k  v \partial^{\alpha} v dx =\frac12R_k$
    is of lower order in $v$. Moreover, note that
\[\Big\|\int_{\R^d} \partial^{\wt{\alpha}} \zeta  \partial^{\beta} v \partial^{\gamma} v dx\Big\|_{\ell^2} \leq \|\zeta\|_{W^{1, \infty}(\Distr;\ell^2)}\int_{\R^d} |\partial^{\beta} v| \, |\partial^{\gamma} v| dx.\]

    From the above it follows that it remains to show  that for every $|\beta|+|\gamma|\leq m-1$.
    \begin{align}\label{eq:oddexpress}
    \frac{\int_{\R^d} |\partial^{\beta} v| \, |\partial^{\gamma} v| dx}{\|v\|_{L^2(\R^d)}} \leq \varepsilon \|v\|_{H^{m}(\R^d)} + C_{\varepsilon} \|v\|_{L^2(\R^d)}
    \end{align}
    By Cauchy--Schwarz' inequality and standard interpolation estimates we find that
    \begin{align*}
    \int_{\R^d} |\partial^{\beta} v| \, |\partial^{\gamma} v| dx& \leq \|v\|_{H^{|\beta|}(\R^d)} \|v\|_{H^{|\gamma|}(\R^d)}
    \\ & \leq C\|v\|_{H^m(\R^d)}^{\frac{|\beta|}{m}} \|v\|_{L^2(\R^d)}^{1-\frac{|\beta|}{m}} \|v\|_{H^m(\R^d)}^{\frac{|\beta_3|}{m}} \|v\|_{L^2(\R^d)}^{1-\frac{|\gamma|}{m}}
    \\ & = C\|v\|_{H^m(\R^d)}^{\ell/m} \|v\|_{L^2(\R^d)}^{2-\frac{\ell}{m}}
    \end{align*}
    where we have set $\ell:=|\beta|+|\gamma|\leq m-1$,
    Therefore, by Young's inequality we obtain that for every $\varepsilon>0$ there exists a $C_{\varepsilon}>0$ such that
    \begin{align*}
    \frac{\int_{\R^d} |\partial^{\beta} v| \, |\partial^{\gamma} v| dx}{\|v\|_{L^2(\R^d)}} \leq C \|v\|_{H^m(\R^d)}^{\ell/m} \|v\|_{L^2(\R^d)}^{1-\frac{\ell}{m}} \leq \varepsilon \|v\|_{H^m(\R^d)} +C_{\varepsilon}\|v\|_{L^2(\R^d)}
    \end{align*}
    which is \eqref{eq:oddexpress}.
\end{proof}

\begin{proof}[Proof of Proposition~\ref{prop:higherorder}]
    Furthermore, set $\phi = \psi = 0$ by Remark~\ref{rem:additive}. We only check coercivity \condref{it:coerc}, since the other conditions are similar to the stochastic heat equation treated in subsections~\ref{stoch_heat_dir} and \ref{stoch_heat_neu} in case of bounded domains. From now on, consider an arbitrary $v\in H^m(\mathbb{R}^d)$. From \eqref{higher_order_A}, we see that
    \begin{equation*}
        2\langle A(t, v), v \rangle = -2 \sum\limits_{|\alpha|, |\beta| = m} \int_{\mathbb{R}^d} A_{\alpha\beta} (\partial^\alpha v) (\partial^\beta v)  \dint x.
    \end{equation*}
    For $\|B(t, v)\|_{L^2(\mathbb{R}^d;\ell^2)}^2$ we obtain
    \begin{align*}
        \|B(t, v)\|_{L^2(\mathbb{R}^d;\ell^2)}^2 & = \sum\limits_{k=1}^\infty \Big\|\sum\limits_{|\alpha|=m} B_{k, \alpha} \partial^\alpha v\Big\|_{L^2(\mathbb{R}^d)}^2                                         \\                                                             & = \sum\limits_{k=1}^\infty \int_{\mathbb{R}^d} \sum\limits_{|\alpha|, |\beta|=m} B_{k, \alpha} B_{k, \beta} (\partial^\alpha v) (\partial^\beta v) \dint x.
    \end{align*}
    The last term that needs to be inspected is $\|B(t, v)^*v\|_{\ell^2}^2$, which is inspected for the cases $m$ odd and $m$ even separately. If $m$ is odd, write $m = 2n + 1$ for $n \in \N_0$.
By Lemma~\ref{lemma: higher_order_coercivity} we obtain
    \begin{equation}
        \frac{\|B(t, v)^*v\|_{\ell^2}^2}{\|v\|_{L^2(\mathbb{R}^d)}^{2}} \leq \varepsilon \|v\|_{H^m(\R^d)}^2 + C_\varepsilon\|v\|_{L^2(\R^d)}^2,
    \end{equation}
    where we are free to choose $\varepsilon > 0$, and $C_\varepsilon$ depends on $B$. Therefore, if $m$ is odd, the following inequalities for the coercivity condition \condref{it:coerc} hold:
    \begin{equation*}
        \begin{split}
            &2\langle A(t, v), v\rangle + \|B(t, v)\|_{L^2(\mathbb{R}^d;\ell^2)}^2 + (p-2)\frac{\|B(t, v)^*v\|_{\ell^2}^2}{\|v\|_{L^2(\mathbb{R}^d)}}\\
            &\leq \sum\limits_{|\alpha|, |\beta| = m} \int_{\mathbb{R}^d} \Big(-2A_{\alpha\beta} + \sum\limits_{k=1}^\infty B_{k, \alpha} B_{k, \beta} \Big) (\partial^\alpha v) (\partial^\beta v) \dint x\\
            &\quad + \varepsilon (p-2)\|v\|_{H^m(\R^d)}^2 + C_\varepsilon (p-2) \|v\|_{L^2(\R^d)}^2\\
            &\leq (-\lambda +\varepsilon) \|v\|_{H^m(\mathbb{R}^d)}^2 + C_\varepsilon \|v\|_{L^2(\R^d)}^2.\\
        \end{split}
    \end{equation*}
    Choosing $\varepsilon$ small enough, the coercivity condition \condref{it:coerc} holds with $\theta = \lambda-\varepsilon(p-2)$, $f = 0$ and $K_c = \varepsilon(p-2)$.

    If $m$ is even, we use the Cauchy-Schwarz inequality to show
    \begin{equation*}
        \frac{\|B(t, v)^*v\|_{\ell^2}^2}{\|v\|_{L^2(\mathbb{R}^d)}^{2}}\leq \|B(t, v)\|_{L^2(\mathbb{R}^d;\ell^2)}^2.
    \end{equation*}
    Using the condition \eqref{eq:wang_coercivity} on the coefficients of Assumptions~\ref{higher_order_assumptions}, we can combine all terms to get the following inequalities for the coercivity condition \condref{it:coerc}:
    \begin{equation*}
        \begin{split}
            &2\langle A(t, v), v\rangle + \|B(t, v)\|_{L^2(\mathbb{R}^d;\ell^2)}^2 + (p-2)\frac{\|B(t, v)^*v\|_{\ell^2}^2}{\|v\|_{L^2(\mathbb{R}^d)}}\\
            &\leq \sum\limits_{|\alpha|, |\beta| = m} \int_{\mathbb{R}^d} \Big(-2A_{\alpha\beta} + (p-1)\sum\limits_{k=1}^\infty B_{k, \alpha} B_{k, \beta} \Big) (\partial^\alpha v) (\partial^\beta v) \dint x\\
            &\leq -\lambda \|v\|_{H^m(\mathbb{R}^d)}^2.\\
        \end{split}
    \end{equation*}
    In this case, the coercivity condition \condref{it:coerc} holds with $\theta = \lambda$, $f = 0$, and $K_c = 0$.
\end{proof}

\subsection{Stochastic p-Laplacian with Dirichlet boundary conditions}\label{stoch_p_lapl}
We consider the following stochastic version of the $p$-Laplace equation:
\begin{equation}\label{stoch_p_laplacian}
    \dint u(t) = \nabla \cdot (|\nabla u(t)|^{\alpha-2} \nabla u(t)) \dint t+ \sum\limits_{k=1}^\infty B_k(u(t)) \dint W_k(t),
\end{equation}
where $(W_k(t))_{t\geq 0}$ are countably many independent Wiener processes. Since we reserve $p$ for the moment in probability, we use $\alpha > 2$ instead of $p$ in the $p$-Laplacian.

We will prove existence, uniqueness and an energy estimate. The arguments are similar to \cite{neelima_2020}, who consider a slightly different leading order operator in \eqref{stoch_p_laplacian}. Moreover, they have an additional nonlinear term $f(u)dt$, which can also be included in our setting.
\begin{assumption}\label{assumptions_p_laplacian}
    Let $\Distr \subset \R^d$ be a bounded domain, $\alpha > 2$, $\gamma^2 \leq 8 \frac{\alpha-1}{\alpha^2}$ and $p \in \big[2, \frac{2}{\gamma^2} + 1\big)$ and $u_0 \in L^{p}(\Omega; L^2(\mathcal{D}))$. Consider
    \[
        (V, H, V^*) = (W_0^{1, \alpha}(\Distr), L^2(\Distr), W^{1, \alpha}_0(\Distr)^*),
    \]
    and set $U = \ell^2$. Let $B \colon W_0^{1,\alpha}(\mathcal D) \to \mathcal L_2(\ell^2,L^2(\mathcal D))$, where for $u \in W^{1,\alpha}_0(\mathcal D)$ and we have $B(u) e_k = B_k(u)$ and $B_k\colon W_0^{1, \alpha}(\Distr) \to L^2(\Distr)$ satisfies $B_k(0) = 0$ and for all $u, v \in W_0^{1, \alpha}(\mathcal{D})$:
    \begin{equation}\label{p-laplace-cond-bk}
        \|B_k(u)-B_k(v)\|^2_{L^2(\mathcal{D})} \leq \gamma_k^2 \| |\nabla u|^{\frac{\alpha}{2}}-|\nabla v|^{\frac{\alpha}{2}}\|_{L^2(\mathcal{D})}^2 + C_k^2\|u-v\|_{L^2(\mathcal{D})}^2,
    \end{equation}
    where we assume $\sum_{k = 1}^\infty \gamma_k \leq \gamma^2$ and $\sum_{k = 1}^\infty C_k^2 < \infty$.
\end{assumption}
Next, we turn SPDE \eqref{stoch_p_laplacian} into a stochastic evolution equation of the form
\[
    \dint u(t) = A(u(t)) \dint t + \sum\limits_{k=1}^\infty B_k(u(t)) \dint W_k(t),
\]
where $A\colon W_0^{1, \alpha}(\mathcal{D}) \to W^{-1, \alpha}(\mathcal{D})$ is given by
\[
    \langle A(u), v\rangle = -\int_{\mathcal{D}}|\nabla u|^{\alpha-2} \nabla u \cdot \nabla v \dint x \qquad \text{for all } u, v \in W_0^{1, \alpha}(\mathcal{D}).
\]

\begin{proposition}
    Given Assumption~\ref{assumptions_p_laplacian}, there exists a unique solution to equation \eqref{stoch_p_laplacian}. Furthermore, there exists a constant $C$ depending on $\gamma$, $\alpha$ and $p$ such that the following estimate holds:
    \begin{equation*}
        \E\sup\limits_{t\in[0, T]} \|u(t)\|_{L^2(\mathcal D)}^p  + \E\Big(\int_0^T \|u(t)\|_{W_0^{1, \alpha}(\mathcal{D})}^\alpha \dint t\Big)^{\frac{p}{2}}
        \leq Ce^{CT}\E\|u_0\|_{L^2(\mathcal{D})}^p.
    \end{equation*}
\end{proposition}
\begin{remark}
    An admissible choice for $B_k$ is $B_k(u) = \gamma_k |\nabla u|^{\frac{\alpha}{2}}$.
\end{remark}

\begin{proof}
    We show that \condref{it:hem}-\condref{it:bound2} hold for equation \eqref{stoch_p_laplacian} and can therefore apply Theorem~\ref{Main_theorem}. Hemicontinuity \condref{it:hem} can be found in \cite[p. 82]{Rockner_SPDE_2015}. For local weak monotonicity \condref{it:weak_mon}, take $u, v \in W_0^{1, \alpha}(\mathcal{D})$ and consider the following inequality which follows from \cite[p. 82]{Rockner_SPDE_2015}:
    \begin{equation}\label{H2_term1}
        2\langle A(u)-A(v), u-v\rangle \leq - 2\int_{\mathcal{D}} \left(|\nabla u|^{\alpha-1}-|\nabla v|^{\alpha-1}\right) (|\nabla u| - |\nabla v|) \dint x
    \end{equation}
    We now consider the other term for \condref{it:weak_mon}. By \eqref{p-laplace-cond-bk} we obtain
    \begin{equation}\label{H2_term2}
        \begin{aligned}
            \|B(u)-&B(v)\|_{\mathcal L_2(\ell^2,L^2(\mathcal D))}^2 \leq \sum\limits_{k=1}^\infty \|B_k(u)-B_k(v)\|^2_{L^2(\mathcal{D})}\\
            &\leq \sum\limits_{k=1}^\infty\gamma_k^2 \| |\nabla u|^{\frac{\alpha}{2}}-|\nabla v|^{\frac{\alpha}{2}}\|_{L^2(\mathcal{D})}^2 + \sum\limits_{i=1}^k C_k^2\|u-v\|_{L^2(\mathcal{D})}^2 \\
            &\leq \gamma^2 \| |\nabla u|^{\frac{\alpha}{2}}-|\nabla v|^{\frac{\alpha}{2}}\|_{L^2(\mathcal{D})}^2 + C \|u-v\|_{L^2(\mathcal{D})}^2.
        \end{aligned}
    \end{equation}
    The bounds \eqref{H2_term1} and \eqref{H2_term2} combine to
    \begin{equation*}
        \begin{split}
            &2\langle A(u)-A(v), u-v\rangle + \|B(u)-B(v)\|^2_{\mathcal L_2(\ell^2,L^2(\mathcal D))} \\
            &\leq -\int_{\mathcal{D}} \big((|\nabla u|^{\alpha-1}-|\nabla v|^{\alpha-1}|) (|\nabla u| - |\nabla v|) + \gamma^2 (|\nabla u |^{\frac{\alpha}{2}}- |\nabla v|^{\frac{\alpha}{2}})^2\big) \dint x\\
            & \quad + C\|u-v\|_{L^2(\mathcal{D})}^2\\
        \end{split}
    \end{equation*}
    Now \condref{it:weak_mon} follows from the inequality (which holds since $\gamma^2 \leq 8 \frac{\alpha-1}{\alpha^2}$) :
    \begin{equation*}
        2\left(x^{\alpha-1}-y^{\alpha-1}\right)(x-y) - \gamma^2 (x^{\frac{\alpha}{2}}-y^{\frac{\alpha}{2}})^2 \geq 0 \quad \text{for all $x,y \geq 0$}.
    \end{equation*}

    In order to show coercivity \condref{it:coerc} note that for $v \in W^{1,\alpha}_0(\mathcal D)$ we have
    \begin{equation*}
        2\langle A(v), v\rangle = -2\int_{\mathcal{D}}|\nabla v |^{\alpha} \dint x = -2\|v\|_{W_0^{1, \alpha}(\mathcal{D})}^\alpha
    \end{equation*}
    and using the Cauchy-Schwarz inequality, we obtain
    \begin{equation*}
        \frac{\|B_t(v)^*v\|_{\ell^2}^2}{\|v\|_{L^2(\mathcal{D})}^2} \leq \|B_t(v)^*\|_{\mathcal L_2(L^2(\mathcal D),\ell^2)}^2 = \|B_t(v)\|^2_{\mathcal L_2(\ell^2, L^2(\mathcal{D}))}.
    \end{equation*}
    Therefore, we conclude with the following $p$-dependent condition for \condref{it:coerc}:
    \begin{equation*}
        \begin{split}
            &2\langle A(v), v\rangle + \|B_t(v)\|_{\mathcal L_2(\ell^2,L^2(\mathcal D))}^2 + (p-2)\frac{\|B_t(v)^*v\|_{\ell^2}^2}{\|v\|_{L^2(\mathcal{D})}^2}\\
            &\leq \left((p-1)\gamma^2-2\right)\|v\|_{W^{1, \alpha}(\mathcal{D})}^\alpha + C\|v\|_{L^2(\mathcal{D})}^2.
        \end{split}
    \end{equation*}
    The first term on the RHS is negative by assumption. Therefore, \condref{it:coerc} holds with $\theta = 2-(p-1)\gamma^2$ and $f = 0$.

    We are only left to show the boundedness conditions \condref{it:bound1} and \condref{it:bound2}. For $v \in W^{1,\alpha}_0(\mathcal{D})$, we use H\"{o}lder's inequality to obtain:
    \begin{align*}
    |\langle A(u), v\rangle| &\leq \Big|\int_{\mathcal{D}} |\nabla u|^{\alpha-2} \nabla u\cdot \nabla v\dint x\Big|
            \\ & \leq \Big(\int_{\mathcal{D}}|\nabla u|^\alpha \dint x\Big)^{\frac{\alpha-1}{\alpha}}\Big(\int_{\mathcal{D}} |\nabla v|^\alpha \dint x\Big)^{\frac{1}{\alpha}}
            \leq \|u\|_{W_0^{1, \alpha}(\mathcal{D})}^{\alpha-1}\|v\|_{W_0^{1, \alpha}(\mathcal{D})}.
        \end{align*}
    Therefore, it follows for all $v \in W^{1, \alpha}_0(\mathcal{D})$ that $\|A(v)\|_{W^{-1, \alpha}(\mathcal{D})}^{\frac{\alpha}{\alpha-1}} \leq \|v\|_{W_0^{1, \alpha}(\mathcal{D})}^\alpha$,
    which entails \condref{it:bound1} with $K_A = \frac 1 2$ and $\beta = 0$. We omit \condref{it:bound2}, since it is clear by assumption.
\end{proof}

\bibliographystyle{plain}
\bibliography{literature}

\begin{thebibliography}{10}

\bibitem{agresti2021stochastic}
A.~Agresti and M.~Veraar.
\newblock {Stochastic Navier-Stokes equations for turbulent flows in critical
  spaces}.
\newblock {\em arXiv preprint arXiv:2107.03953}, 2021.

\bibitem{AV19_QSEE_1}
A.~Agresti and M.C. Veraar.
\newblock Nonlinear parabolic stochastic evolution equations in critical spaces
  {P}art {I}. {S}tochastic maximal regularity and local existence.
\newblock {\em Accepted for publication in Nonlinearity, arXiv preprint
  arXiv:2001.00512}, 2020.

\bibitem{AV19_QSEE_2}
A.~Agresti and M.C. Veraar.
\newblock Nonlinear parabolic stochastic evolution equations in critical spaces
  {P}art {II}. {B}low-up criteria and instantaneous regularization.
\newblock {\em Accepted for publication in Journal of Evolution Equations,
  arXiv preprint arXiv:2001.00512}, 2020.

\bibitem{AV20_NS}
A.~Agresti and M.C. Veraar.
\newblock {Stochastic Navier-Stokes equations for turbulent flows in critical
  spaces}.
\newblock {\em arXiv preprint arXiv:2107.03953}, 2020.

\bibitem{AHMT}
P.~Auscher, S.~Hofmann, A.~McIntosh, and P.~Tchamitchian.
\newblock The {K}ato square root problem for higher order elliptic operators
  and systems on {$\Bbb R^n$}.
\newblock {\em J. Evol. Equ.}, 1(4):361--385, 2001.
\newblock Dedicated to the memory of Tosio Kato.

\bibitem{bensoussan_equations_1972}
A.~Bensoussan and R.~Temam.
\newblock Equations aux derivees partielles stochastiques non lineaires.
\newblock {\em Israel Journal of Mathematics}, 11(1):95--129, March 1972.

\bibitem{brzezniak_1991}
Z.~Brze\'{z}niak, M.~Capi\'{n}ski, and F.~Flandoli.
\newblock Stochastic partial differential equations and turbulence.
\newblock {\em Math. Models Methods Appl. Sci.}, 1(1):41--59, 1991.

\bibitem{brzezniak_strong_2014}
Z.~Brze\'{z}niak, W.~Liu, and J.~Zhu.
\newblock Strong solutions for {SPDE} with locally monotone coefficients driven
  by {L{\'e}vy} noise.
\newblock {\em Nonlinear Analysis: Real World Applications}, 17:283--310, June
  2014.

\bibitem{veraar_2012}
Z.~Brze\'{z}niak and M.C. Veraar.
\newblock Is the stochastic parabolicity condition dependent on {$p$} and
  {$q$}?
\newblock {\em Electron. J. Probab.}, 17:no. 56, 24, 2012.

\bibitem{carlen_1991}
E.~Carlen and P.~Kr\'{e}e.
\newblock {$L^p$} estimates on iterated stochastic integrals.
\newblock {\em Ann. Probab.}, 19(1):354--368, 1991.

\bibitem{DaPrato_1994}
G.~Da~Prato, A.~Debussche, and R.~Temam.
\newblock Stochastic {B}urgers' equation.
\newblock {\em NoDEA Nonlinear Differential Equations Appl.}, 1(4):389--402,
  1994.

\bibitem{du_2020}
K.~Du, J.~Liu, and F.~Zhang.
\newblock Stochastic {H}\"{o}lder continuity of random fields governed by a
  system of stochastic {PDE}s.
\newblock {\em Ann. Inst. Henri Poincar\'{e} Probab. Stat.}, 56(2):1230--1250,
  2020.

\bibitem{KimLee}
Kyeong-Hun Kim and K.~Lee.
\newblock A note on {$W_p^\gamma$}-theory of linear stochastic parabolic
  partial differential systems.
\newblock {\em Stochastic Process. Appl.}, 123(1):76--90, 2013.

\bibitem{krylov_stochastic_1981}
N.~V. Krylov and B.~L. Rozovskii.
\newblock Stochastic evolution equations.
\newblock {\em Journal of Soviet Mathematics}, 16(4):1233--1277, July 1981.

\bibitem{krylov_1979}
N.~V. Krylov and B.~L. Rozovski\u{\i}.
\newblock Stochastic evolution equations.
\newblock In {\em Current problems in mathematics, {V}ol. 14 ({R}ussian)},
  pages 71--147, 256. Akad. Nauk SSSR, Vsesoyuz. Inst. Nauchn. i Tekhn.
  Informatsii, Moscow, 1979.

\bibitem{rockner_2010}
W.~Liu and M.~R\"{o}ckner.
\newblock S{PDE} in {H}ilbert space with locally monotone coefficients.
\newblock {\em J. Funct. Anal.}, 259(11):2902--2922, 2010.

\bibitem{Rockner_SPDE_2015}
W.~Liu and M.~R\"{o}ckner.
\newblock {\em Stochastic partial differential equations: an introduction}.
\newblock Universitext. Springer, Cham, 2015.

\bibitem{Motron02}
M.~Motron.
\newblock Around the best constants for the {S}obolev trace map from
  {$W^{1,1}(\Omega)$} into {$L^1(\partial\Omega)$}.
\newblock {\em Asymptot. Anal.}, 29(1):69--90, 2002.

\bibitem{neelima_2020}
Neelima and D.~\v{S}i\v{s}ka.
\newblock Coercivity condition for higher moment a priori estimates for
  nonlinear {SPDE}s and existence of a solution under local monotonicity.
\newblock {\em Stochastics}, 92(5):684--715, 2020.

\bibitem{pardoux_equations_1975}
E.~Pardoux.
\newblock {\em Équations aux dérivées partielles stochastiques non
  linéaires monotones: étude de solutions fortes de type {Ito}}.
\newblock PhD thesis, 1975.
\newblock OCLC: 489811603.

\bibitem{Ren08}
Y.-F. Ren.
\newblock On the {B}urkholder-{D}avis-{G}undy inequalities for continuous
  martingales.
\newblock {\em Statist. Probab. Lett.}, 78(17):3034--3039, 2008.

\bibitem{rozovskii_1990}
B.~Rozovskii.
\newblock {\em {Stochastic Evolution Systems}}.
\newblock Springer, 1990.

\bibitem{wang_2020}
Y.~Wang and K.~Du.
\newblock Schauder-type estimates for higher-order parabolic {SPDE}s.
\newblock {\em J. Evol. Equ.}, 20(4):1453--1483, 2020.

\end{thebibliography}

\end{document}